\documentclass[11pt]{amsart}
\usepackage{amssymb,amsmath,latexsym,enumerate, dsfont,ifpdf}
\usepackage{graphicx}
\usepackage{float}
\usepackage{color}
\ifpdf
\usepackage{hyperref}
\else
\usepackage[hypertex]{hyperref}
\fi
\renewcommand{\phi}{\varphi}

\newtheorem{theorem}{Theorem}[section]

\newtheorem{lemma}[theorem]{Lemma}
\newtheorem{corollary}[theorem]{Corollary}

\newtheorem{defi}[theorem]{Definition}
\newenvironment{definition}{\begin{defi} \rm}{ \end{defi}}
\newtheorem{exa}[theorem]{Example}

\newtheorem{rem}[theorem]{Remark}
\newenvironment{remark}{\begin{rem} \rm}{ \end{rem}}
\newtheorem{prob}[theorem]{Problem}
\newenvironment{problem}{\begin{prob} \rm}{ \end{prob}}

\DeclareMathOperator{\range}{range}

\DeclareMathOperator{\Thm}{Thm}

\newcommand{\Ax}{\operatorname{Ax}}

\newcommand{\set}[1]{\{{#1}\}}

\usepackage{pifont}
\def\G1{\hbox{$\displaystyle{\mbox{\ding{172}}}$}}

\title[Trial and error mathematics:
dialectical systems and completions of theories]{Trial and error mathematics:\\
dialectical systems and completions of theories}

\author[J.~Amidei]{Jacopo Amidei}
\address{School of Computing \& Communications\\
The Open University\\
Milton Keynes, UK}
\email{jacopo.amidei@open.ac.uk}

\author[U.~Andrews]{Uri Andrews}
\address{Department of Mathematics\\
University of Wisconsin\\
Madison, WI 53706-1388\\
USA}
\email{\href{mailto:andrews@math.wisc.edu}{andrews@math.wisc.edu}}
\urladdr{\url{http://www.math.wisc.edu/~andrews/}}

\author[D.~Pianigiani]{Duccio Pianigiani}
\address{Dipartimento di Ingegneria Informatica e Scienze Matematiche\\
Universit\`a Degli Studi di Siena\\
I-53100 Siena, Italy}
\email{duccio.pianigiani@unisi.it}

\author[L.~San Mauro]{Luca San Mauro}
\address{Institute for Discrete Mathematics and Geometry\\
Vienna University of Technology\\
1040 Vienna, Austria}
\email{luca.san.mauro@tuwien.ac.at}
\urladdr{\url{http://dmg.tuwien.ac.at/sanmauro/}}

\author[A.~Sorbi]{Andrea Sorbi}
\address{Dipartimento di Ingegneria Informatica e Scienze Matematiche\\
Universit\`a Degli Studi di Siena\\
I-53100 Siena, Italy}
\email{andrea.sorbi@unisi.it}

\keywords{Dialectical system, $q$-dialectical system, $p$-dialectical system, completion}

\subjclass[1991]{03A99, 03D55, 03D99}

\thanks{The authors wish to thank the anonymous referee of a version of this paper,
in particular for her/his precious comments and suggestions on the history and philosophy of ``trial
and error mathematics''.
San Mauro was partially supported
by the Austrian Science Fund FWF through projects P~27527 and M~2461.
Sorbi is a member of the INDAM-GNSAGA group.}

\begin{document}

\begin{abstract}
This paper is part of a project that is based on the notion of dialectical
system, introduced by Magari as a way of capturing trial and error
mathematics. In \cite{trial-errorsI} and \cite{trial-errorsII}, we
investigated the expressive and computational power of dialectical systems,
and we compared them to a new class of systems, that of quasidialectical
systems, that enrich Magari's systems with a natural mechanism of revision.
In the present paper we consider a third class of systems, that of
$p$-dialectical systems, that naturally combine features coming from the
two other cases. We prove several results about $p$-dialectical systems and
the sets that they represent. Then we focus on the completions of
first-order theories. In doing so, we consider systems with connectives,
i.e.  systems  that encode the rules of classical logic. We show that any
consistent system with connectives represents the completion of a given
theory. We prove that dialectical and $q$-dialectical systems coincide with
respect to the completions that they can represent. Yet, $p$-dialectical systems are
more powerful: we exhibit a $p$-dialectical
system representing a  completion of Peano Arithmetic which is neither dialectical
nor $q$-dialectical.
\end{abstract}

\maketitle

\section{Introduction and background}
Formal systems represent mathematical theories in a rather static way, in
which axioms of the represented theory have to be defined from the beginning,
and no further modification is permitted. It has often been argued that this
representation is not comprehensive of all aspects of real mathematical
theories: see, for instance, the seminal work of
Lakatos~\cite{Lakatos:Proofs}
for arguments against any hastily correspondence between formal systems and
the way in which mathematicians deal with real theories.
Our goal is to model cases in which a mathematician, when defining a new
theory, chooses axioms through some trial and error process, instead of
fixing them, once for all, at the initial stage. A possible way of
characterizating such cases is provided by the so-called experimental logics,
firstly studied by Jeroslow in the 1970's~\cite{Jeroslow:experimental} (for a
nice  discussion about these logics and their philosophical meaning the
reader is referred to K\aa sa~\cite{Kasa:PhD}).  Our approach is based on
another notion, that of dialectical systems, introduced by
Magari~\cite{Magari:SucerteTeorie} in the same period.  In doing so, we
continue the investigations initiated in \cite{trial-errorsI} and
\cite{trial-errorsII}.

The basic ingredients of \emph{dialectical  systems} are a number $c$,
encoding a contradiction; a deduction operator $H$, that tells us how to
derive consequences from a finite set of statements $D$; and a proposing
function $f$, that proposes statements to be accepted or rejected as
provisional theses of the system. In \cite{trial-errorsI}, we introduced a
new class of systems, that of \emph{$q$-dialectical systems} (there called
``quasidialectical''), by  enriching Magari's systems with a natural
mechanism of revision. This is obtained by means of two additional
ingredients: a replacement function $f^-$, that provides for all axioms a
substituting axiom, and a symbol $c^-$, that encodes any sort of problem,
possibly weaker than mathematical contradiction, that can justify the
replacement of a certain axiom. In \cite{trial-errorsI} and
\cite{trial-errorsII}, we drew an accurate comparison between the
expressivness of these two systems. In particular, we showed the following:
dialectical sets and $q$-dialectical sets (i.e., the set of statements that
are eventually accepted by, respectively, dialectical and $q$-dialectical
systems) are always $\Delta^0_2$; the two systems have the same computational
power, in the sense that the class of Turing degrees that contains a
dialectical set coincide with the class of Turing degrees that contains a
$q$-dialectical set (and in fact they are equivalent to the class of
computably enumerable Turing degrees); yet, $q$-dialectical sets form a class
which is much larger than that of dialectical sets, since $q$-dialectical
sets  inhabit each level of Ershov hierarchy, while dialectical sets are all
$\omega$-computably enumerable.

In this paper we consider a third class of systems, called
\emph{$p$-dialectical systems}. Their introduction is motivated by two main
reasons. First, $p$-dialectical systems naturally combines, in their
behaviour, features characterizing the other two classes: they have a
mechanism of revision, as in the case of $q$-dialectical systems, but they do
not distinguish between $c$ and $c^-$, having only $c$ in their syntax, as in
the case of dialectical systems. In fact, dialectical and $q$-dialectical
systems can be defined as modifications of $p$-dialectical systems, as we
will do in the next section. The second important reason  for focusing on
this new class is connected with the completions of first-order theories.

As is shown below, if we restrict to the case of \emph{systems with
connectives}, i.e.\ systems in which the deduction operator $H$ has to
satisfy the rules of classical logic, then we do obtain the following: if $S$
is a system that does  not derive the contradiction from the empty set of
premises, then $S$ is the \emph{completion} of a given theory. We make use of
this fact to compare the expressiveness of our systems regarded as machines
to build, in the limit, completions. We show that dialectical and
$q$-dialectical completions coincide, all lying in the class of $\omega$-c.e.
sets. On the contrary, $p$-dialectical systems are much more powerful: for
every effectively indexed class of $\Delta^{0}_{2}$ sets we exhibit a
concrete example of a $p$-dialectical system representing a $p$-dialectical
set which is a completion of Peano Arithmetic not lying in that class. Each
such $p$-dialectical system  can be also looked at as an example of how a
$p$-dialectical system works in concrete, perhaps the first such examples
even considering \cite{trial-errorsI,trial-errorsII}, more concerned with
laying down the theoretical bases rather than examples and applications.

We would like to remark at this point that although dialectical systems may
be viewed as a possible approach to trial and error mathematics, the emphasis
in this paper is of a rather abstract nature. More than on the adequacy of
these systems to formalize trial and error mathematics, we are mainly
interested in the computability theoretic properties of the dialectical sets
(i.e. the sets represented by these systems), and in the use of dialectical
systems and our suggested variations of the dialectical procedure as tools
for producing more and more complicated $\Delta^0_2$ completions of
consistent formal theories having a strong enough expressive power. In
Section~\ref{sct:comparison} however we sketch a brief comparison between
dialectical systems and other approaches based on knowledge or assumptions
revision.

Although the exposition of this paper is rather self-contained, a certain
familiarity with the definitions of dialectical and $q$-dialectical systems,
as presented in  \cite{trial-errorsI}, might help the reader that aims at
fully understanding the behaviour of the $p$-dialectical systems we will
introduce next. Our computable theoretic notions are standard and as in
Soare~\cite{Soare:Book}.

\subsection{The $p$-dialectical systems}
A $p$-dialectical system  shall be thought as a machine for constructing a
theory in stages, by adjusting the set of axioms whenever a contradiction is
derived. This is the same intuition that both dialectical and $q$-dialectical
systems aim at modelling (see \cite{trial-errorsI}). What distinguishes the
three cases is how they respond to the emergence of a contradiction, and
whether they are allowed to revise an axiom, when this is temporary rejected
by system, instead of being forced to fully dismiss it. We shall begin with
the formal definition of $p$-dialectical systems.

In what follows, if $f$ is the so-called proposing function, we will denote
$f(i)$ with $f_i$.

\begin{definition}\label{def:$p$-systems}
A \emph{$p$-dialectical system} is a quadruple $p=\langle
H, f, f^-,c\rangle$, where
\begin{enumerate}
\item $H$ is an enumeration operator such that $H(\emptyset)\ne \emptyset$,
    $H(\{c\})=\omega$, and $H$ is an algebraic closure operator, i.e., $H$
    satisfies, for every $X \subseteq \omega$,
\begin{itemize}
  \item $X \subseteq H(X)$;
  \item $H(X) \supseteq H(H(X))$.
\end{itemize}
\item $f$ is a computable permutation of $\omega$;
\item $f^-$ is an \emph{acyclic} computable function, i.e.,~for every $x$, the
    $f^{-}$-orbit of $x$, i.e. the set
    \[
    \set{x, f^-(x), f^-(f^-(x)),
    \ldots, (f^-)^n(x), \ldots},
    \]
    is infinite.
\end{enumerate}

We call $f$ the \emph{proposing function}, $f^-$ the \emph{revising
function}, $c$ \emph{the contradiction}.
\end{definition}

\subsection*{The $p$-dialectical procedure}
Given such a $p=\langle H,f,f^-,c\rangle$, and starting from a a computable
approximation $\alpha=\set{H_s}_{s \in \omega}$ (i.e.\ a computable sequence
of finite sets, given by their canonical indices, such that $H_s \subseteq
H_{s+1}$ and $H=\bigcup_s H_s$), define by induction values for several
computable parameters, which depend on $\alpha$: $A_s$ (a finite set), $r_s$
(a function such that for every $x$, $r_{s}(x)$ is a finite string of
numbers, viewed as a vertical string, or \emph{stack}), $m(s)$ (the greatest
number $m$ such that $r_s(m)\ne \langle \mbox{ }\rangle$, where the symbol
$\langle \mbox{ }\rangle$ denotes the empty string). 
In
addition, there are the derived parameters: $\rho_{s}(x)$ is the top of the
stack $r_{s}(x)$, $L_{s}(x)=\set{\rho_{s}(y): y<x \text{ and } r_{s}(y)\ne
\langle \mbox{ } \rangle}$, and, for every $i$, $\chi_{s}(i)=
H_{s}(L_{s}(i+1))$.

\subsubsection*{Stage $0$}
Define $m(0)=0$,
\[
r_{0}(x)=
\begin{cases} \langle f_0 \rangle &\text{$x=0$} \\
\langle \mbox{ } \rangle &\text{$x>0$},
\end{cases}
\]
and let $A_0=\emptyset$.

\subsubsection*{Stage $s+1$} Assume $m(s)=m$. We distinguish the following
cases:

\begin{enumerate}
\item there exists no $k \le m$ such that $\{c\} \cap \chi_{s}(k)\ne
    \emptyset$: in this case, let $m(s+1)=m+1$, and define
\[
r_{s+1}(x)=
\begin{cases} r_{s}(x) &\text{if $x \leq m$} \\
\langle f_{m+1} \rangle &\text{if $x = m+1$} \\
\langle \mbox{ } \rangle &\text{if $x > m+1$}; \\
\end{cases}
\]

\item there exists $k\le m$ such that $c\in \chi_{s}(k)$: in this case, let
    $z$ be the least such $k$, let $m(s+1)=z$, 
    and define, where $\rho_{s}(z)=f_y$,
\[
r_{s+1}(x)=
\begin{cases} r_{s}(x) &\text{$x < z$} \\
r_{s}(x)^{\smallfrown} \langle f^{-}(f_{y}) \rangle &\text{$x = z$} \\
\langle \mbox{ } \rangle &\text{$x >  z+1$}. \\
\end{cases}
\]
\end{enumerate}
Finally  let
\[
A_{s+1}=\bigcup_{i < m(s+1)}\chi_{s+1}(i).
\]
Notice that $A_{s+1}=H_{s+1}(L_{s+1}(m(s+1)))$ if $m(s+1)>0$; otherwise
$A_{s+1}=\emptyset$.

Figures~\ref{fig:pfigure1} and \ref{fig:pfigure2} 
illustrate how we go from stage s to stage s + 1, according to
(1) and (2), respectively, of the definition. The vertical strings above
the various slots represent the various stacks $r(x)$ at the given stage. In
each figure, only the relevant slots are depicted.

\begin{figure}[h!]
\begin{center}
\includegraphics[scale=1]{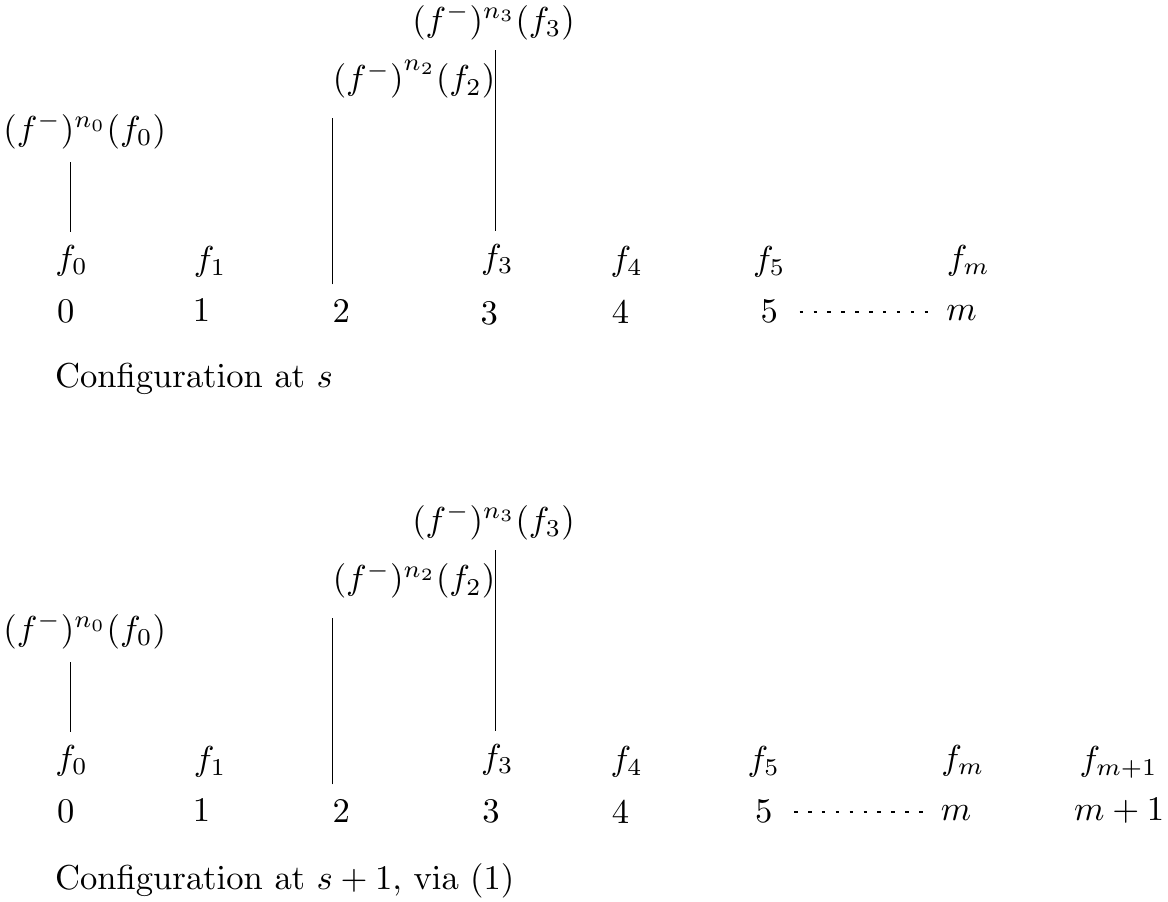}
\end{center}
\caption{From stage $s$ to $s+1$ using (1). }\label{fig:pfigure1}
\end{figure}

\begin{figure}[h!]
\begin{center}
\includegraphics[scale=1]{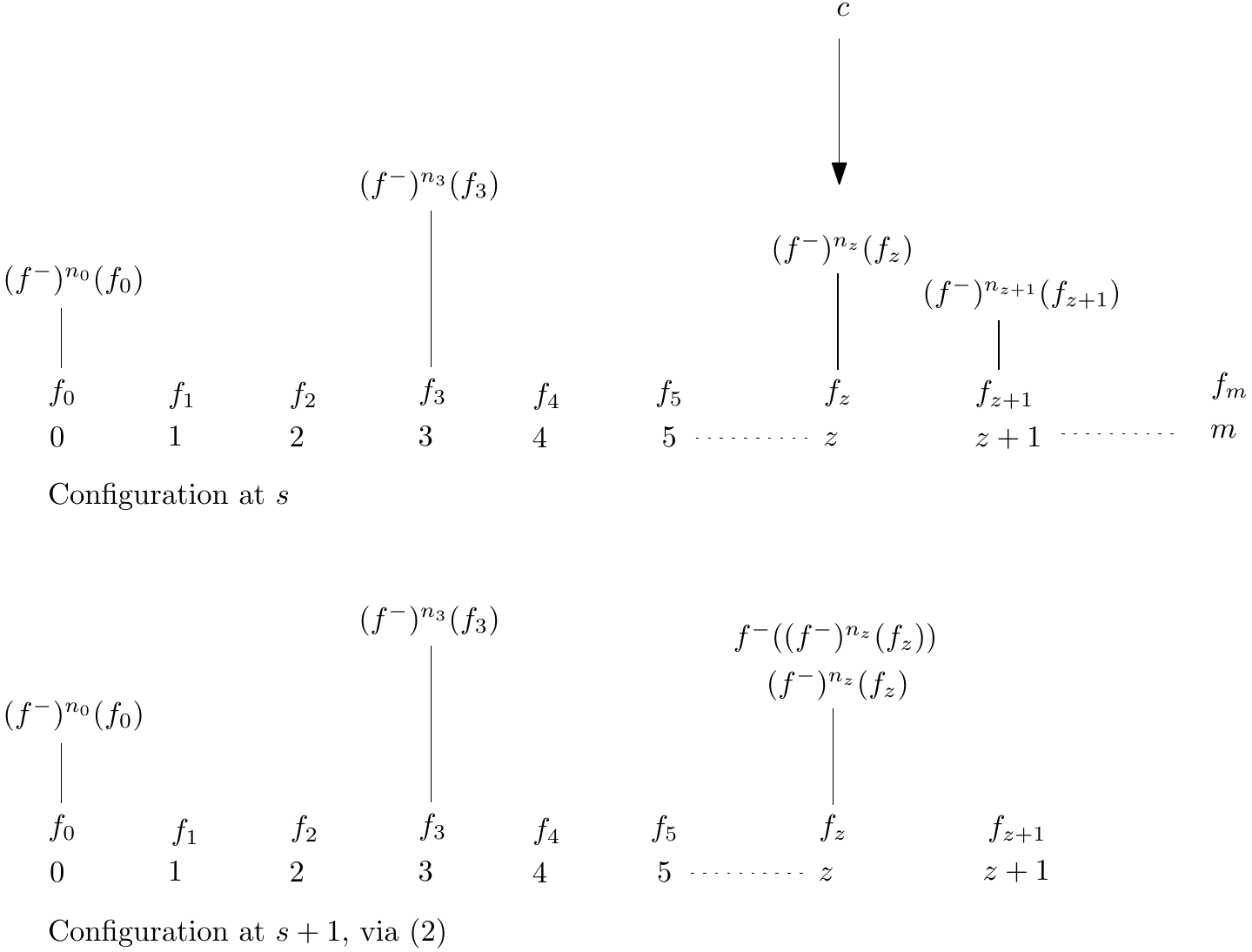}
\end{center}
\caption{From stage $s$ to $s+1$ using (2).}\label{fig:pfigure2}
\end{figure}

We say that a $p$-dialectical system with enumeration operator $H$ is
\emph{consistent} if $c\notin H(\emptyset)$. We call $A_s$ the set of
\emph{provisional theses} of $p$ with respect to $\alpha$ at stage $s$. The
set $A_p$ defined as
\[
A_{p}=\set{f_x: (\exists t)(\forall s \ge t)[f_{x}\in A_{s}]}
\]
is called the set of \emph{final theses} of  $p$: notice that we write $A_p$
and not $A_p^\alpha$ because we are going to show in next theorem that this
set does not in fact depend on the approximation.

In the following theorem and its proof, relatively to any given approximation
$\alpha$ we agree that $\lim_s r_s(u)$ \emph{exists finite} if there exists
$t$ such that $r_s(u)=r_t(u)$ for all $s \ge t$; $\lim_s r_s(u)$ \emph{exists
infinite} if there exists a stage $t$, such that for all $s \ge t$ $r_s(u)$
is an initial segment of $r_{s+1}(u)$ and $\bigcup_{s\ge t} r_{s}(u)$ is an
infinite string; finally we say that $\lim_s r_s(u)$ \emph{does not exist} if
for infinitely many $s$ we have $r_s(v)=\langle \mbox{  }\rangle$.

\begin{lemma}\label{lem:independence}
The set of final theses of a $p$-dialectical system does not depend on the
chosen approximation of the enumeration operator $H$, and, independently of
the approximation, for every $u$, either $\lim_s r_s(u)$ exists finite, or
$\lim_s r_s(u)$ exists infinite and in this case for every $v>u$ $\lim_s
r_s(v)$  does not exist; or $\lim_s r_s(u)$ does not exist, and in this case
also for every $v>u$ $\lim_s r_s(v)$  does not exist.
\end{lemma}

\begin{proof}
Let $p=\langle H, f, f^-, c\rangle$ be a $p$-dialectical system and
$\alpha=\{H_s\}$ an approximation to $H$.

First of all, if $\lim_s r_s(u)$ exists infinite, the every time we redefine
$r_s(u)$ we also set $r_s(v)=\langle \mbox{  }\rangle$ for every $v>u$;
moreover it is easy to see that if $\lim_s r_s(v)$  does not exist then there
is some $u<v$ such that $\lim_s r_s(u)$ exists infinite.

So the claim about $\lim_s r_s(u)$ amounts to show that either $\lim_s
r_s(u)$ exists finite for every $u$ or there is a least $u$ such that $\lim_s
r_s(u)$ exists infinite.

Now, $L(0)=\lim_s L_s(0)=\emptyset$ and clearly this value does not depend on
the approximation. Suppose that $L(u)$ reaches limit and the limit does not
depend on the approximation, and let us consider $u+1$: we show by induction
on $i$ that after $L(u)$ has reached limit, the $i$-th bit $r(u)_i$ is the
same (whether defined or undefined) whatever approximation one considers.
Now, $r(u)_0= \langle f_u \rangle$ whatever the approximation; and clearly,
assuming the claim for $i$, we have that $r(u)_{i+1}=f^-(r(u)_{i})$ if and
only if $c \in H(L(u)\cup \{r(u)_{i}\})$, which shows independence from the
approximation. In particular $\lim_s r_s(u)$ exists either finite or
infinite, independently of the approximation.
\end{proof}

\begin{theorem}\label{thm:independence}
As granted by the previous lemma, let $u$ be the greatest number $\le \omega$
such that the limit value $L(u)$ exists finite, i.e. for every $v<u$ $\lim_s
r_s(v)$ exists finite. If $u>0$ then, independently of the approximation,
$A_p =H(L(u))$ (where $L(\omega)= \bigcup_{v \in \omega}L(v)$ if $u=\omega$),
and $A_p=L(\omega)$ if $u=\omega$. If $u=0$ then, independently of the
approximation, $A_p=\emptyset$.
\end{theorem}

\begin{proof}
Let $u$ be as in the statement of the theorem. Let us show first that
$H(L(u)) \subseteq A_p$. If $u<\omega$ then cofinitely many times we have
$L(u) \subseteq L_s(m(s))$, which implies $H_s(L(u)) \subseteq H_s(L_s(m(s))$;
hence cofinitely many times we have $H_s(L(u)) \subseteq A_s$, which implies
that
$H(L(u))
\subseteq A_p$, independently of the approximation. If $u=\omega$ then for
every $v$ an argument similar to the previous case shows that $L(v)\subseteq
H_s(L(v)) \subseteq A_s$ for cofinitely many $s$, and thus $L(\omega) \subseteq
A_p$, independently of the approximation.

We want now to show now that $A_p \subseteq H(L(u)))$ and $A_p \subseteq
L(\omega)$ if $u=\omega$: in the latter case by properties of $H$ this implies
also that $A_p \subseteq H(L(\omega))$. We distinguish again the two possible
cases:
\begin{itemize}
  \item $u<\omega$: for infinitely many $s$, we have that (whatever the
      approximation) $A_s=H_s(L(u))$ which shows that $A_p \subseteq
      H(L(u))$.
  \item $u=\omega$: suppose now that $f_x \notin L(\omega)$. Then, whatever
      the approximation, $c \in H(L(x) \cup \{f_x\})$, and thus $f_x \notin
      H_s(L_s(x))$ for every big enough stage. Let $t$ be a stage such that
      starting from this stage $L(x)$ has reached limit already and $f_x
      \notin H_s(L_s(x))$ for every $s \ge t$. Let $t_0 \ge t$ be such that
      $f_x\in H_{t_0}(L_{t_0}(v))$ ($v=m(t_0)$: notice that $x<v$); then
      there is a stage $s_1 \ge t_0$ such that $c \in H_{s_1}(L_{t_0}(v))$;
      it follows that there is a stage $t_1 \ge t_0$ such that $L(v)$
      changes value at $t_1$, giving $m(t_1)< v$. If $f_x \notin A_{t_1}$
      then we have found a stage $\ge t_0$ at which$f_x \notin A_p$ has
      changed; otherwise we repeat the same argument, but taking
      $v=m(t_1)$. By choice of $t$ and properties of $x$, it is clear that
      proceeding in this way we end up with some $t' \ge t_0$ such that
      $f_x \notin A_{t'}$. We have shown that for every $t_0$ such that
      $f_x\in A_{t_0}$ there is a later stage $t'$ such that $f_x\notin
      A_{t'}$. As this works for whatever approximation we use, this shows
      that $f_x \notin A_p$ whatever the approximation. We have thus shown
      that $A_p \subseteq L(\omega)$.
\end{itemize}
Finally we consider the case $u=0$. In this case $m(s)=0$ infinitely many
times, then $A_{p}=\emptyset$, whatever the approximation.
\end{proof}

\begin{definition}
A pair $(p,\alpha)$ where $p=\langle H, f, f^-, c\rangle$ is a
$p$-dialectical system and $\alpha$ is an approximation to $H$ is called
\emph{loopless} if for every $u$, the set $\{\rho_s(u): s \in \omega\}$ is
finite.
\end{definition}

\begin{remark}
In view of the previous theorem if there is a loopless approximation then all
approximations are loopless, and we will be justified in talking about a
\emph{loopless} $p$-dialectical system, and referring to the \emph{final
theses} $A_p$ of $p$, without mentioning any special approximation $\alpha$
to the enumeration operator $H$ of $p$.
\end{remark}

\begin{corollary}
If $p$ is loopless then $A_p=L(\omega)$.
\end{corollary}

\begin{proof}
See the proof of Theorem~\ref{thm:independence}.
\end{proof}

A set $A \subseteq \omega$ is called \emph{p-dialectical} if
$A=A_p$ for some $p$-dialectical system, and we say in this case that $A$ is
\emph{represented} by $p$.

\subsection{Dialectical systems and $q$-dialectical systems}

Dialectical systems and \emph{$q$-dialectical systems} have been extensively
studied in \cite{trial-errorsI} and \cite{trial-errorsII}; the reader is
referred to these papers for both full definitions of them and philosophical
motivations for their study. For our present interests, let us show where the
definition of a $p$-dialectical system is to be modified in order to obtain
these others systems.

\begin{definition}\label{def:reviseddefinition}
A  \emph{dialectical system} is a $p$-dialectical system with no revising
function. That is to say, a dialectical system is a triple $d=\langle H, f,c
\rangle$, in which $H,f,c$ satisfy the same conditions formulated within
Definition \ref{def:$p$-systems}. All the others parameters we have
introduced for $p$-dialectical systems
($A_s,r_s,m(s),\rho_s(x),r_s(x),L_s(x)$, and $\chi_s(i)$) hold the same
meaning for dialectical systems.

\subsection*{Dialectical procedure}
The dialectical procedure is equal verbatim to the $p$-dialectical procedure
for stage $0$, and for any application of Clause $(1)$ of any given stage
$s+1$. Thus the only difference is with Clause $(2)$, which in the case of
dialectical systems has to be modified as follows:

\begin{itemize}
\item[(2)] there exists $k\le m$ such that $c\in
    \chi_{s}(k)$: in this case, let $z$ be the least such $k$, and
    distinguish two cases:
\smallskip

\begin{itemize}
\item[(2.1)] if $c \in H_s(\emptyset)$, then let
    $m(s+1)=0$, and define
\[
r_{s+1}(x)=
\begin{cases} \langle f_0\rangle &\text{if $x=0$} \\
\langle \mbox{  }\rangle &\text{if $x>0$}; \\
\end{cases}
\]
\item[(2.2)]
otherwise, let  $m(s+1)=z+1$, and define
\[
r_{s+1}(x)=
\begin{cases} r_{s}(x) &\text{if $x < z$} \\
\langle f_{z+1}\rangle  &\text{if $x=z+1$} \\
\langle \mbox{  }\rangle &\text{if $x=z$ or $x>z+1$}. \\
\end{cases}
\]
\end{itemize}
\end{itemize}

\end{definition}
We say that a dialectical system with enumeration operator $H$ is
\emph{consistent} if $c\notin H(\emptyset)$. The sets of final theses of
dialectical systems and dialectical sets are defined in a similar way to
$p$-dialectical systems.

Let us then move to $q$-dialectical system. A $q$-dialectical system,
intuitively, incorporates both distinguishing features of dialectical and
$p$-dialectical systems, in the sense that some axiom $f_x$ can be either
discarded, as in the case of a dialectical system, or revised by $f^-$, as in
the case of a dialectical system. Since the formal defintion of a
$q$-dialectical system (that can be found in \cite{trial-errorsI}) for the
most part is identical to that of a $p$-dialectical system, we limit
ourselves to point the differences between the two.

\begin{definition}
A $q$-dialectical system is a quintuple $q= \langle H,f,f^-,c,c^- \rangle$,
such that $ \langle H,f,c \rangle$ is a dialectical system, $f^-$ satisfies
the condition expressed for a $p$-dialectical system, and finally $c^-\in
\omega \smallsetminus \range(f^-)$.

We call $c^-$ the \emph{counterexample}.

\end{definition}

\subsection*{$q$-dialectical procedure}
Stage $0$ of the $q$-dialectical procedure is identical to the same stage of
both the $p$-dialectical and the dialectical procedure. Concerning stage
$s+1$, we have now three different clauses instead of two (the additional one
being introduced since we deal with both $c$ and $c^-$):

\begin{enumerate}
\item there exists no $k \le m$ such that $\{c, c^{-}\} \cap \chi_{s}(k)\ne
    \emptyset$: in this case, let $m(s+1)=m+1$, and define
\[
r_{s+1}(x)=
\begin{cases} r_{s}(x) &\text{if $x \leq m$} \\
\langle f_{m+1} \rangle &\text{if $x = m+1$} \\
\langle \mbox{ } \rangle &\text{if $x > m+1$}; \\
\end{cases}
\]

\item[(2)] there exists $k\le m$ such that $c\in
    \chi_{s}(k)$, and for all $k'<k$, $c^- \notin \chi_{s}(k')$: in this
    case, let $z$ be the least such $k$m and distinguish two cases:
\smallskip

\begin{itemize}
\item[(2.1)] if $c \in H_s(\emptyset)$, then let
    $m(s+1)=0$, and define
\[
r_{s+1}(x)=
\begin{cases} \langle f_0\rangle &\text{if $x=0$} \\
\langle \mbox{  }\rangle &\text{if $x>0$}; \\
\end{cases}
\]
\item[(2.2)]
otherwise, let  $m(s+1)=z+1$, and define
\[
r_{s+1}(x)=
\begin{cases} r_{s}(x) &\text{if $x < z$} \\
\langle f_{z+1}\rangle  &\text{if $x=z+1$} \\
\langle \mbox{  }\rangle &\text{if $x=z$ or $x>z+1$}. \\
\end{cases}
\]
\end{itemize}

\item[(3)] there exists $k\le m$ such that $c^{-}\in \chi_{s}(k)$, and for
    all $k'\le k$, $c \notin \chi_{s}(k')$: in this case, let $z$ be the
    least such $k$, let $m(s+1)=z$,
    and
    define, where $\rho_{s}(z)=f_y$,
\[
r_{s+1}(x)=
\begin{cases} r_{s}(x) &\text{$x < z$} \\
r_{s}(x)^{\smallfrown} \langle f^{-}(f_{y}) \rangle &\text{$x = z$} \\
\langle \mbox{ } \rangle &\text{$x >  z+1$}. \\
\end{cases}
\]
\end{enumerate}
Finally define
\[
A_{s+1}=\bigcup_{i < m(s+1)}\chi_{s+1}(i).
\]
Thus $A_{s+1}=\emptyset$ if $m(s+1)=0$, and
$A_{s+1}=H_{s+1}(L_{s+1}(m(s+1)))$ otherwise.

As is clear, Clause (1) is almost identical to the
Clause (1) of the $p$-dialectical procedure; Clause (2) is essentially the
same of Clause (2) of the dialectical procedure; Clause (3) is essentially
the same of Clause (2) of the $p$-dialectical procedure.

We say that a $q$-dialectical system, with enumeration operator $H$, is
\emph{consistent} if $\set{c,c^-}\cap H(\emptyset)=\emptyset$. We call $A_s$
the set of \emph{provisional theses} of $q$ with respect to $\alpha$ at stage
$s$. The set $A_q^{\alpha}$ defined as
\[
A_{q}^{\alpha}=\set{ f_x: (\exists t)(\forall s \ge t)[f_{x}\in A_{s}]}
\]
is called the set of \emph{final theses} of  $q$ with respect to $\alpha$. We
often write $A_{s}=A^{\alpha}_{q,s}$ when we want to specify the
$q$--dialectical system $q$ and the chosen approximation to the enumeration
operator. A pair $(q, \alpha)$ as above is called an \emph{approximated
$q$-dialectical system}. A set $A \subseteq \omega$ is called
\emph{$q$-dialectical} if $A=A^{\alpha}_q$ for some approximated
$q$-dialectical system, and we say in this case that $A$ is
\emph{represented} by the pair $(q,\alpha)$.

We summarize some of the main properties of $A_d$ and $A_q^{\alpha}$. As in
the case of $p$-dialectical systems, we say that an approximated
$q$-dialectical system is \emph{loopless} if the set $\{\rho^s(u): s \in
\omega\}$ is finite, for all $u$. For more information and properties about
loopless approximated $q$-dialectical system, and in particular for a
complete characterization of approximated $q$-dialectical system with loops,
the reader is referred to \cite{trial-errorsI}.

\begin{theorem}[\cite{trial-errorsI, Magari:SucerteTeorie}]\label{thm:summary}
If $d$ and $(q, \alpha)$ are respectively a dialectical system or a loopless
approximated $q$-dialectical system then the following hold:
\begin{enumerate}
  \item $A_d$ and $A_q^{\alpha}$ are $\Delta^0_2$ sets;
  \item for every $x$, $\lim_s r_s(x)=r(x)$ and $\lim_s L_s(x)=L(x)$ exist
      finite (whether the functions $r_s(x)$, $L_s(x)$ refer to $d$, or
      $(q,\alpha)$) and
\begin{align*}
  A_d&=\set{f_x: r(x)=\{f_x\}}\\
  A_q^{\alpha}&=\set{f_x: r(x)=\langle f_x \rangle},
\end{align*}
and
\begin{align*}
  f_x \in A_d &\Leftrightarrow c \notin H(L(x) \cup \{f_x\})\\
  f_x \in A_q^{\alpha} &\Leftrightarrow \{c, c^-\} \cap H(L(x)
          \cup \{f_x\})=\emptyset.
\end{align*}
\end{enumerate}
(For $q$-dialectical systems, the values of $r$ and $L$ depend in general on
the chosen approximation $\alpha$).
\end{theorem}

\begin{proof}
The claim that $A_d$ is a $\Delta^0_2$ set comes from
\cite{Magari:SucerteTeorie}. The other claims come from \cite[Lemma~3.8,
Lemma~3.18]{trial-errorsI} (to show that $A^\alpha_q$ is $\Delta^0_2$ see
also the proof of \cite[Lemma~3.4]{trial-errorsII} which amends a previous
bug in \cite{trial-errorsI}).
\end{proof}

Notice that for a $p$-dialectical system, being loopless implies being
consistent.

Most of the results proved for $q$-dialectical sets extend to $p$-dialectical
sets. In particular,

\begin{theorem}\label{thm:summaryp-dialectical}
If $p$ is a loopless $p$-dialectical system then
$\lim_{s} L_{s}(x)$ exists for every $x$ and
\[
f_x \in A_p \Leftrightarrow c \notin H(L(x) \cup \{f_x\}).
\]
\end{theorem}
\begin{proof}
The proof follows from Theorem~\ref{thm:independence}, and an easy induction.
Following the last stage at which $L(x)$
ceases to change, we propose $r_s(x)=\langle f_x \rangle$, and it is easy to
see that
\[
r(x)=\langle f_x \rangle \Leftrightarrow c \notin H(L(x) \cup \{f_x\}).
\]
\end{proof}

Notwithstanding the independence of $L$ from the chosen approximation to $H$
established in Lemma \ref{lem:independence} and Theorem
\ref{thm:independence} nothing guarantees that the sequence $\{A_s\}_{s \in
\omega}$ of sets of provisional theses is independent of the approximation,
or does even give a $\Delta^0_2$ approximation to $A_p$. The following lemma
shows however that from any given $H$ one can find an approximation for which
the sequence $\{A_s\}_{s \in \omega}$ is in fact a  $\Delta^0_2$
approximation to $A_p$.

\begin{lemma}\label{lem:goodapproximation}
If $H$ is an algebraic closure operator then from any computable
approximation to $H$ we can effectively find an approximation $\{\hat H_{s}:
s \in \omega\}$ to an enumeration operator $\hat{H}$ such that for every $s$,
the enumeration operator given by $\hat H_{s}$ is an algebraic closure
operator (more precisely it satisfies $X \subseteq \hat H_{s}(X)$ if $\max X
\le s$, and $\hat H_{s}(\hat H_{s}(X)) \subseteq \hat H_{s}(X))$ for all
$X$), and $H$ and $\hat{H}$ coincide as enumeration operators, i.e. for every
$X$, $\hat{H}(X)=H(X)$.
\end{lemma}

\begin{proof}
Given any enumeration operator $G$, we can effectively find a closure
operator $G^\omega$ which extends $G$: the details of this construction can
be found for instance in \cite{trial-errorsI}. Moreover if $G \subseteq K$
then $G^\omega \subseteq K^\omega$; $G$ is a closure operator if and only if
(as enumeration operators, not as c.e. sets) $G=G^\omega$; if $G$ is finite
then $G^\omega$ is finite and the canonical index of $G^\omega$ can be
effectively computed from that of $G$. Suppose now that $\{H_{s}: s \in
\omega\}$ be a computable approximation to a closure operator $H$: we may
assume that the approximation satisfies
\begin{enumerate}
\item if $\langle x, D\rangle \in H_s$ then $x, \max D<s$;
\item for every $i<s$, $\langle i, \{i\}\rangle \in H_s$.
\end{enumerate}
For every $s$ define $\hat{H}_{s}=(H_s)^\omega$. By the above remarks, this
is a full-fledged computable approximation to $H^\omega$, still satisfying
(1) and (2). But (as enumeration operators, not as c.e. sets) $H=H^\omega$,
as $H$ is a closure operator. So $\{\hat H_{s}: s \in \omega\}$ is the
desired approximation, effectively found from $\{H_{s}: s \in \omega\}$, to a
suitable closure operator $\hat{H}$ (namely $\hat{H}=H^\omega$) which
coincides as an operator with $H$.
\end{proof}

The next definition summarizes the properties of the approximation built in
the proof of the previous theorem.

\begin{definition}
If $H$ is an algebraic closure operator and $\{H_{s}\}$ is a computable
approximation to it, we say that the approximation is \emph{good} if for
every $s$ the following hold:  $X \subseteq H_{s}(X)$ if $\max X \le s$, and
$H_{s}(H_{s}(X)) \subseteq H_{s}(X))$ for all $X$.
\end{definition}

\begin{corollary}
If $p=\langle H, f, f^-, c\rangle$ is a $p$-dialectical system, and
$\{H_{s}\}$ is a good approximation to $H$ then the corresponding
$p$-dialectical approximation $\{A_{s}: s \in \omega\}$, given by the
$p$-dialectical procedure, is a $\Delta^0_2$ approximation.
\end{corollary}

\begin{proof}
If $p$ is not consistent then the claim follows from the fact that starting
from the stage at which $c \in H(\emptyset)$ we have that $m_{s}(0)=0$ and
thus $A_{s}=\emptyset$.

If $p$ is consistent then we can use Theorem~\ref{thm:summaryp-dialectical}.
Let $f_u=x$, and assume that $x \notin A_p$. Let $t_0$ be a stage such that
$L(u)$ has already reached limit $L(u)$. As $x \notin A_p$, we have that $c
\in H(L(u) \cup \{x\})$: let $t_1\ge t_0$ be such that $L(u) \subseteq
L_{s}(m(s))$ for every $s \ge t_1$ and  $c \in H_{t_1}(L(u) \cup \{x\})$, and
suppose that $s>t_1$ is a stage such that $x \in A_{s}$, i.e.\ $x \in
L_{s}(m(s))$ and $s>  x, \max (L_u)$. It follows that $L(u) \subseteq
H_s(L(u)) \subseteq H_{s}(L_{s}(m(s)))$ and $\{x\} \subseteq
H_{s}(L_{s}(m(s)))$, hence $L(u) \cup \{x\} \subseteq H_{s}(L_{s}(m(s)))$,
hence by goodness of the approximation, $H_s(L(u)\cup \{x\}) \subseteq
H_{s}(L_{s}(m(s)))$, giving that $c \in H_{s}(L_{s}(m(s)))$, contradicting
the definition of $m(s)$.
\end{proof}


\section{Comparing dialectical sets, $p$-dialectical sets,
and $q$-dialectical sets}

In this section we compare under inclusion the notions of $p$-dialectical
system, dialectical system, and $q$-dialectical system. Throughout the
section we will use superscripts appended to the parameters $L, r, \rho$ etc.
(for instance $L^{p}, L^{d}, L^{q}$ or $r^{p}, r^{d}, r^{q}$) to distinguish
whether the parameters refer to the $p$-dialectical system, or the
dialectical system, or the $q$-dialectical system we will happen to be
talking about.

\begin{theorem}\label{thm:from-d-to-p}
Given any dialectical system $d=\langle H, f,c,\rangle$ such that $H(\emptyset)$
is infinite, we can build a $p$-dialectical system $p$ such that $A_d=A_p$.
\end{theorem}

\begin{proof}
Let $d=\langle H,f,c\rangle$, and being $H(\emptyset)$ an infinite c.e. set,
let $Z=\set{z_0 < z_1< \ldots < z_i< \ldots} \subseteq H(\emptyset)$ be a
computable set. Then, let $p$ be the $p$-dialectical system $p=\langle
H,f,f^-,c\rangle$ where
\[
f^-(x)=
\begin{cases}
z_0 &\ x\notin Z, \\
z_{i+1} &\ x=z_{i}\in Z.
\end{cases}
\]
This definitions obeys the requirement that the orbits of $f^{-}$ be
infinite. Now, we know that $\lim_{s}r_{s}^{d}(u)$ exists for every $u$.
Notice that for every $X$ and for every $z\in H(\emptyset)$, from $H$ being a
closure operator it follows that
\[
c \in H(X\cup \{z\}) \Leftrightarrow c \in H(X).
\]
Using this, it is easy to show by induction on $u$ that
\begin{itemize}
\item  if $r^{d}(u)=\langle f_{u} \rangle$ then $r^{p}(u)=\langle
    f_{u}\rangle$, and if $r^{d}(u)=\langle \mbox{ } \rangle$ then
    $r^{p}(u)=\langle f_{u}, z_{0} \rangle$.
\end{itemize}
It follows that
\[
\bigcup_{u} L^{p}(u)=\bigcup_{u} L^{d}(u)\cup \{z_{0}\}.
\]
$A_{d}=H(\bigcup_{u} L^{p}(u))$ (see \cite{trial-errorsI}, but the proof is
similar to the proof of Theorem~\ref{thm:independence}). On the other hand,
by Theorem~\ref{thm:independence}
\[
A_{p}=H(\bigcup_{u} L^{p}(u))=H(\bigcup_{u} L^{d}(u)
   \cup \{z_{0}\})=H(\bigcup_{u}  L^{d}(u))
\]
because $z_{0}\in H(\emptyset)$ and $H$ is an algebraic closure operator.
\end{proof}

\begin{theorem}\label{thm:from-p-to-q}
Any $p$-dialectical set is a $q$-dialectical set. In fact, given a
$p$-dialectical system $p$ we can effectively build a $q$-dialectical system
$q$ such that $A_p=A_q^\alpha$ for any approximation $\alpha$ to the operator
of $q$.
\end{theorem}

\begin{proof}
Let $p=\langle H,f,f^-,c\rangle$. We first observe that the claim is trivial
if $A_p=\omega\smallsetminus \{c\}$, and if $p$ has loops.

If not, let $u_0$ be the least number such that $z_0=f_{u_0} \neq c$ and $z_0
\notin A_p$, and denote $\rho_p(u_0)$ with $z_1$. Consider the
$q$-dialectical system $q=\langle H^*,f^*,f^-,z_0,c^-\rangle$, where $c^-=c$,
$f^*$ is defined as follows
\[
f^*(x)=\begin{cases} z_1 &\text{ if $x=0$},\\
f(x-1) &\text{ if $x>0$},
\end{cases}
\]
and \[
H^{*}=(H\smallsetminus \set{\langle z_0, D\rangle
\mbox{ $:$ } z_0\notin D}) \cup \{\langle x, \set{z_0}\rangle: x \in \omega\}.
\]
Notice that for every set $X$, if $z_0 \in H^*(X)$ then $z_0 \in X$.

We now show that $H^{*}$ is an algebraic closure operator.
\begin{itemize}
\item We first show that $X\subseteq H^{*}(X)$ for every set $X$.  Let $X$
    be given. If $z_{0} \in X$, we have that $X \subseteq \omega \subseteq
    H^{*}(X)$. If $z_{0}\notin X$ and $x \in X$ then (as $H$ is an
    algebraic closure operator) there  is an axiom $\langle x, D\rangle \in
    H$ with $D\subseteq X$, but then then $\langle x, D\rangle \in H^{*}$
    as well and thus $x \in H^{*}(X)$.

\item Next we show that $H^{*}(H^{*}(X))\subseteq  H^{*}(X) $.   Let $X$ be
    given, and assume that $x \in H^{*}(H^{*}(X))$. We  may also assume
    that $z_{0} \notin H^{*}(H^{*}(X))  \cup H^{*}(X) \cup X$, otherwise in
    any case $z_{0}\in X$ by definition of $H^{*}$ and thus
    $H^{*}(H^{*}(X))\subseteq \omega \subseteq H^{*}(X)$.

    So assume that $x \ne z_{0}$ and let $\langle x, D\rangle \in H^{*}$ be
an axiom with $D \subseteq H^{*}(X)$: to this axiom by our assumptions
(which imply $z_{0} \notin D$) must correspond an axiom $\langle x, D
\rangle \in H$. For every $y \in D$ there is an axiom $\langle y,
E_{y}\rangle \in H^{*}$ with $E_{e} \subseteq X$ and by our assumptions
again,  each such axiom must correspond to an axiom $\langle y,
E_{y}\rangle \in H$. We thus obtain $x \in H(H(X))$, and since $H$ is an
algebraic closure operator, this gives $x\in H(X)$ via an axiom, say,
$\langle x, E\rangle \in H$: but this is also an axiom of $H^{*}$, thus $x
\in H^{*}(X)$.
\end{itemize}

Let us now work with any approximation $\alpha$ to $H^*$. We want now to
prove that $A_p=A_q^\alpha$. In particular, we show by induction on $u$ that,
for all $u$, we have that
\[
r^q(u)=
\begin{cases}
\langle z_1 \rangle, &\text{if $u=0$},\\
r^p(u-1), &\text{if $u>0$ and $z_0 \notin \range(r^p(u-1))$},\\
\in \{\langle \mbox{ }\rangle, r^p(u-1)\}, &\text{otherwise},
\end{cases}
\]
where the third clause means that $r^q(u)=\langle \mbox{ }\rangle$ or
$r^q(u)= r^p(u-1)$ depending on which one between $z_0$ and $c$ appears
first, enumerated in $H^*(L^q(u)\cup \{z_0\})$, at the relevant stage of the
$q$-dialectical procedure. Moreover, we show by induction on $u>0$ that
\[
r^q(u)= \langle \mbox{ }\rangle \Rightarrow \rho_p(u-1)=z_1,
\]
so that $L^q(u)=L^p(u-1)\cup \{z_1\}$.

Since $f^*_0=z_1$, it is immediate to notice that that $r^q(0)=z_1$. Indeed,
we can not have $z_{0} \in H^{*}(\{z_{1}\})$ by definition of $H^{*}$, but we
cannot have $c \in H^*(\set{z_1})$ either, otherwise $c \in H(\set{z_1})$
against the fact that $z_1 \in A_p$.

Then  consider the case $u>0$, and assume by induction that
$L^q(u)=L^p(u-1)$. It is easy to see that if $z_0 \in \range(r^p(u-1))$ then
$\rho^p (u-1)=z_1$. Suppose that $r^{p}(u-1)$ has length $n$: we claim that
for every $i<n$, $(r^{q}(u))_{i}=(r^{p}(u-1))_{i}$, and $\rho_{q}(u)
=\rho_{p}(u-1)$. This is clearly true when $i=0$ by definition of $f^{*}$.
Assume the claim is true of $i<n-1$. If $(r^{p}(u-1))_{i}=(r^{q}(u))_{i}\ne
z_0$, then (as $z_0 \notin L^q(u)\cup \{(r^{q}(u))_{i}\}$ by induction), we
have that $z_0 \notin H^*(L^q(u)\cup \{(r^{q}(u))_{i}\})$; but (since
$i<n-1$) $c \in H(L^p(u-1)\cup \{(r^{p}(u-1))_{i}\})$, thus $c \in
H^*(L^q(u)\cup \{(r^{q}(u))_{i}\})$ (by the way $H^*$ is defined), hence
\[
r^{q}(u)_{i+1}=f^{-}((r^{q}(u)))_{i}=f^{-}((r^p{u-1})_{i})=r^{p}(u-1)_{i+1}.
\]
On the other hand, when we reach the top, $c \notin H(L^{p}(u-1)\cup
\{\rho_{p}(u-1)\}\cup \{z_{1}\})$, and thus again $\{z_0, c\} \cap H^*(L^q(u)
\cup \{\rho_{p}(u-1)\})= \emptyset$, giving that $\rho_{q}(u)=\rho_{p}(u-1)$.

Let us consider now the case $(r^{p}(u-1))_{i}=(r^{q}(u))_{i}= z_0$. Now both
$\{z_0,c\}\subseteq H^*(L^q(u)\cup \{(r^{q}(u))_{i}\})$. If at the relevant
stage of the $q$-dialectical procedure, $\alpha$ shows $z_{0}$ derivable from
$H^*(L^q(u)\cup \{(r^{q}(u))_{i}\})$ no later than $c$ is so derivable, then
$r^{q}(u)=\langle \mbox{ }\rangle$ and $\rho_p(u-1)=z_1$; if $\alpha$ shows
$c$ derivable first, then by an argument similar to the one for the case when
$(r^p(u-1))_i\ne z_0$, we conclude that $r^{q}(u)_{i+1}=r^{p}(u-1)_{i+1}$.
Since $f^-$ is not cyclic, we now have that $z_0 \ne (r^p(u-1))_j$ for all $i
<j<n$, thus again as in the case seen above when $(r^p(u-1))_i\ne z_0$, we
conclude that for all $i<j<n$, $r^{q}(u)_{j}=r^{p}(u-1)_{j}$, and eventually
$\rho_{q}(u)=\rho_{p}(u-1)$.

It follows that $L^p(\omega)=L^q(\omega)$, and thus $A_p=A_q^\alpha$ for
every approximation $\alpha$ to $H^*$.
\end{proof}

The next problem is left open.

\begin{problem}
Are there $q$-dialectical sets that are not $p$-dialectical?
\end{problem}

\section{A brief comparison with other approaches to trial and error
mathematics}\label{sct:comparison}

Having set the formal definitions of the three systems (dialectical,
$q$-dialectical, and $p$-dialectical systems) and laid down the theoretical
bases, before moving to a detailed investigation of the computability
theoretic properties of the sets they represent, including certain
completions of formal theories, it is perhaps time to pause and briefly
compare these systems with other popular models of trial and error
mathematics.

\subsection{Belief revision}

The central problems facing the theory of belief revision are how to revise a
knowledge system in the light of new information that  turns out to be
inconsistent with the old one. The AGM axiomatic theory~\cite{AGM} is the
most famous theory of belief revision: in this model, beliefs are represented
as sentences held by an agent. Such sentences form a deductive closed set: a
\emph{belief set}. To formalize how agents revise their beliefs, AGM
describes various actions by which  a belief set can be modified in response
to new information. If this new information does not contradict the set of
acquired knowledge, it is simply added to the belief set and we have the
\textit{expansion}. On the contrary, \textit{revision} takes place when a new
sentence turns out to be inconsistent with the belief set to which it is
added. In order to maintain consistency, some of the old sentences are
deleted by an action called \textit{contraction}. What is kept of the old
beliefs is the consequence of some guiding rules. Two \emph{dogmas}, in
particular, have been singled out (see \cite{Rott} for more details): first,
one's prior beliefs should be changed as little as possible; second, whenever
there is a choice about which sentence should be deleted, the agent should
abandon the least one with respect to some ordering of \textit{epistemic
entrenchment}, where ``$q$ is more entrenched than $p$'' intuitively means
that the sentence $q$ has more epistemic value than the sentence $p$. So, the
overall goal of these dogmas is to keep the loss of information minimal when
a belief set is updated.

Dialectical systems, and the variations considered in this paper, aim at
modeling similar actions, but they implement them in a rather different way.
In this context, expansion is not limited to the addition of a new sentence
(or axiom, in our terminology)  but it consists also in increasing the
deductive power of the deduction operator $H$ (whereas in AGM each action
leads to an already deductively closed set of beliefs).

More importantly, the dialectical model lacks an explicit entrenchment
ordering: when a conflict emerges, i.e., $c$ or $c^-$ is derived, we
reject/revise the last proposed axiom of the minimal inconsistent set,
instead of evaluating the epistemic value of the axioms contained in it.
Nevertheless, the behavior of the proposing function $f$ and that of the
revising function $f^-$ to some extent surrogate the entrenchment: $f$
encodes a certain priority to the axioms to be proposed, and $f^-$ (in the
case of $p$- and $q$- dialectical systems) can dynamically change this
priority by swapping the ordering of two given axioms and thus modifying
their mutual priority. One might go further and develop a dialectical model
where to each axiom is assigned a certain weight: whenever a conflict arises,
the system keeps as provisional theses the consistent subset $X$ of the old
knowledge that realizes the maximum weight. A similar line of research has
been explored in \cite{MSS}, where the authors investigate generalized
dialectical systems embodied with probability weights. Yet, also this
approach differs from the AGM proposal, since entrenchment is more concerned
with  the \textit{explanatory power} of the sentences. In G\"ardenfors' and
Makinson's words~\cite{Makinson}:

\begin{quotation}
Rather than being connected with probability, the epistemic entrenchment of a
sentence is tied to its explanatory power and its overall informational value
within the belief set. For example, lawlike sentences generally have greater
epistemic entrenchment than accidental generalizations. This is not because
lawlike sentences are better supported by the available evidence (normally
they are not) but because giving up lawlike sentences means that the theory
loses more of its explanatory power than giving up accidental
generalizations.
\end{quotation}

Studying dialectical systems that incorporate some measures of explanatory
power (as the ones discussed for instance in \cite{explanatory}) is a topic
for future work.

\subsection{Lakatos' philosophy of mathematics}

It would be incorrect to assert that dialectical systems attempt to formalize
Lakatos' philosophy of mathematics: the dialectical model is way too abstract
to offer a convincing rendering of the dynamic of  mathematical discovery
characterized, e.g., in \cite{Lakatos:Proofs}. Yet, Lakatos' intuition that
mathematical knowledge is subject to constant refinement motivates Magari's
original proposal. Indeed, according to Magari~\cite{Magari:SucerteTeorie}, a
dialectical system is best understood as modeling a mathematician (or even, a
mathematical community) that in developing a mathematical theory proceeds by
trial and errors, instead of merely accumulating more and more deductions (as
classical formal systems prescribe).  Moreover, the main conceptual reason for
moving from dialectical to $q$-dialectical systems in \cite{trial-errorsI} was
precisely that of including in our systems a revision mechanism more adherent
to that of mathematical practice, rather than just limiting ourselves to
logical contradiction.

To sketch a more precise parallel between our systems and Lakatos' approach,
it is worth to briefly contrast the dialectical model with the way in which
Lakatos' theory has been computationally represented: in \cite{Peaseetal},
the authors make use of abstract argumentation systems (in Dung-style, see
\cite{dung1995acceptability}) to offer an automated realization of Lakatos'
view. In the field of structured argumentation (the interested reader is
referred to \cite{besnard2014introduction}), an \textit{abstract
argumentation framework} is a directed graph, where the nodes are
\textit{arguments} and the arcs are \textit{attacks}, and a set of arguments
is conflict-free if no pair of argument belongs to the set of attacks. An
\textit{argument} \textit{system} is then given by a logical language, a set
of rules (that can be either \textit{strict} or \textit{defeasible}), and a
partial function from rules to formulas. In a nutshell, Lakatos' account is
represented in \cite{Peaseetal} as a \textit{formal dialogue game} between a
Proponent and an Opponent (roles that are possibly embodied by many speakers)
and proofs are carefully represented as arguments that correspond to the
artifacts collaboratively created by the participants in a Lakatosian
dialogue, such as the one  famously exemplified by the classroom  debate
about Euler's conjecture on polyhedra in \cite{Lakatos:Proofs}. This dialogue
game is a rather complex game, in which players can perform different types
of moves (such as raising counterexamples, piecemeal exclusion, monster
barring, monster adjusting, etc.), corresponding to crucial ingredients of
Lakatos' informal logic.

The dialectical model is of course way less adherent to Lakatos' perspective.
A game-theoretic formulation of it can
however be readily obtained: the Proponent makes a proposal via the function
$f$ and, and at each step of the computation, the
Opponent tries to reject by either proving its inconsistency or its
implausibility with acquired knowledge. So, the game can be roughly intended
as a debate between the Proponent and
the Opponent about whether any given sentence is to be accepted or not.
However, such a game is much more rigid than
the one formulated in \cite{Peaseetal}. For instance, unlike Lakatos' game
where the roles are
interchangeable, in our models  the Opponent always attacks
and the Proponent always proposes new hypotheses. Another major difference is
that Lakatos' game does not contain
strict rules (i.e, rules of the form ``$B$ is is always a consequence of
$A$''), but only
defeasible rules (i.e, rules of the form ``typically $B$ is a consequence of
$A$''). On the contrary, no defeasible reasoning is allowed in the dialectical
game: in fact, an argument can be attacked only by showing some
undesirable \textit{deductive} consequences, and this depends only on the set
of premises and the deduction operator.

Finally, the strategy of the Proponent and the Opponent are completely
deterministic, being defined  once for all at the beginning of the computation
and eventually producing a unique set of final theses (modulo the
approximation to $H$ in the case of the $q$-dialectical systems). This is why
our analysis is centered on the class of sets represented by the $p$- or $q$-
dialectical systems,  rather than focusing on the behavior of a particular
system.

\subsection{Algorithmic learning theory}

Algorithmic learning theory (ALT) is a vast research program, initiated by
Gold~\cite{gold1967language} and Putnam~\cite{putnam1965trial} in
the 60s that comprises different models of learning in the limit. It deals
with the
question of how a \emph{learner}, provided with more and more data about some
\emph{environment},
is eventually able to achieve systematic knowledge about it. For instance, a
classic paradigm in ALT concerns the learning of total computable functions:
the learner receives as input the stream of values of a function $g$ to be
learned and, at any stage, outputs a conjecture of a program that computes the
function. The learning is successful if the learner eventually infer a correct
program for $g$. Different formalizations of this and similar intuitions gave
rise to a vast research area (for an introduction to the field see for
instance \cite{Jain:Book}).

In analogy with the learning criteria explored in ALT, a dialectical system
also embeds a stabilization process, by which we eventually converge to a set
of final theses (and in fact, by Theorem \ref{thm:summary} and Theorem
\ref{thm:summaryp-dialectical} we have that, if a set is represented by our
system, then it is computable in the limit, i.e., $\Delta^0_2$). More
importantly, the existence of a similar stabilization mechanism hints at a
deeper similarity between the two models: they both display and manage
information essentially \emph{by stages}, in a way that is naturally apt to
be analyzed by computable theoretic tools. The significance of this common
trait is well described by the following remark of Van Benthem in
\cite{VANBENTHEM}:

\begin{quotation}
Perhaps the key activity tied up with theory change is learning, whether by
individuals or whole communities. Modern learning theory (...) describes
learning procedures over time, as an account of scientific methods in the
face of steadily growing evidence, including surprises contradicting one's
current conjecture. In this perspective, update, revision, and contraction
are single steps in a larger process, whose temporal structure needs to be
brought out explicitly (...). Learning theory is itself a child of recursion
theory, and hence it is one more illustration of a computational influence
entering philosophy.
\end{quotation}

Dialectical systems, and our related models, are children of recursion theory
as well. They do not offer a \emph{logic} of trial and error mathematics, nor
do they aim at spelling out a variety of principles by which we might want to
change or preserve a given axiom. This can be seen as a limitation of
dialectical systems. But note that no logic of learning (or of inductive
inference) is provided in ALT, and no axiomatization of computability is
contained in Turing's 1936 paper \cite{Turing1936}. This is because the
emphasis of a computable theoretic investigation (such as the present one) is
typically more process-oriented and focuses on exploring the computational
costs of such processes. Dialectical, $p$-dialectical, and $q$-dialectical
systems are attempts at characterizing the evolution of abstract mathematical
theories by defining highly idealized agents that follow few mechanic rules
-- by which, nonetheless, a rich class of theories can be produced. One might
insist that such an idealization is too extreme; in fact, in this section we
offered enough evidence that other frameworks might give a better
understanding of, e.g., what belief change is. Yet, a measure of the
fruitfulness of a given idealization also comes from whether it sheds new
light on some well-established notion. The goal of the second half of this
paper is to show that, for the dialectical model, this is exactly the case:
our systems turn out to be a remarkably good machinery for dealing with a
key-concept of classical logic, i.e., completions of first-order theories.

\section{Systems with connectives and completions}
By a \emph{system} we will mean in general a $p$-dialectical system or a
dialectical system, or a $q$-dialectical system. From now on we will restrict attention
to systems in which, via identification of numbers with the sentences of some
formal language, $H$ is ragarded as a logical deduction operator, i.e. $H(X)$
is the set of sentences which can be logically derived from the premises $X$.
In this identification sentential connectives can be viewed as just
computable functions.

The following definition is taken from \cite{Magari:SucerteTeorie}.

\begin{definition}\label{def:withconnectives}
A \emph{system with connectives}  is a system with
an enumeration operator $H$, a contradiction $c$, and injective computable
functions $\neg,
\rightarrow, \wedge, \lor$ such that
for every $X\subseteq \omega$ and $x,y \in \omega$,
\begin{enumerate}
\item $c \in H(\{x, \neg x\})$;
\item $H(\{\neg \neg x\})=H(\{x\})$;
\item $x \lor \neg x \in H(\emptyset)$;
\item $H(X\cup \{x \lor y\})=H(X\cup \{x\})\cap H(X\cup \{y\})$;
\item if $c\in H(X\cup \{x\})$ then $\neg x \in H(X)$;
\item $x \in H(X\cup \{y\})$ if and only if $y \rightarrow x \in H(X)$.
\end{enumerate}
\end{definition}

\begin{definition}
Given a system with connectives and finale theses $A$, we say that the system
is a \emph{completion}, if for every  $x$, $A\cap \{x, \neg x\}$ has exactly
one element.
\end{definition}

\subsection{$q$-dialectical completions}
It is known from \cite{trial-errorsII} that there are (loopless)
$q$-dialectical sets that are not dialectical. Unfortunately if we consider
connectives, nothing is gained from passing from dialectical systems to
$q$-dialectical systems.

We first show that if a loopless $q$-dialectical system with connective is
consistent (i.e. $\set{c,c^-} \cap H(\emptyset)=\emptyset$, where $H$ is the
operator of $q$), then $A_{q}$ is a completion.

\begin{theorem}\label{thm:q-implies-completion}
If $q=\langle H, f,f^-,c,c^{-}\rangle$ is a consistent loopless $q$-dialectical
system with connectives, $\alpha$ an approximation to $H$ such that
$(q,\alpha)$ is loopless, then $A_{q}^{\alpha}$ is a completion.
\end{theorem}

\begin{proof}
Let $q, \alpha$ be as in the statement of the theorem; for simplicity, let us
write $A_{q}=A^{\alpha}_{q}$

Assume now that $x$ is the least number such that $f_{x} \notin A_{q}$, and
$\neg f_{x} \notin A_{q}$: let $f_{y}=\neg f_{x}$, and assume without loss of
generality that $y<x$, the other case $x<y$ being similar. By
Theorem~\ref{thm:summary} this is the consequence of one of the following
circumstances:
\begin{enumerate}

\item $c \in H(L(x) \cup \{f_{x}\})$, and $c \in H(L(y) \cup \{\neg
    f_{x}\})$: hence, $c\in H(L(x) \cup \{f_{x}\})$ and $c\in H(L(y) \cup
    \{\neg f_{x}\})$, and by (d) of Definition~\ref{def:withconnectives},
    we have that $c \in H(L(x) \cup \{f_{x} \lor \neg f_{x}\})$. But then,
    as $f_{x} \lor \neg f_{x} \in H(\emptyset)$, we have $c \in H(L(x))$,
    contrary to the fact that $L(x)$ is the limit set.

\item $c \in H(L(x) \cup \{f_{x}\})$, and $c^{-} \in H(L(y) \cup \{\neg
    f_{x}\})$: in this case, it is easy to see (under the assumption that
    $y<x$) that $c^{-}\in H(L(x) \cup \{\neg f_{x}\})$ and $c^{-}\in H(L(x)
    \cup \{f_{x}\})$, giving that $c^{-}\in H(L(x) \cup \{f_{x} \lor \neg
    f_{x}\})$, and thus  $c^{-} \in H(L(x))$, contrary to the fact that
    $L(x)$ is the limit set.

\item $c^{-} \in H(L(x) \cup \{f_{x}\}$, and $c \in H(L(y) \cup \{\neg
    f_{x}\}$: the argument is simlar, having this time (under the
    assumption $y<x$) $c \in H(L(x) \cup \{\neg f_{x}\})$, and thus $c \in
    H(L(x))$.

\item $c^{-} \in H(L(x) \cup \{f_{x}\})$, and $c^{-} \in H(L(y) \cup \{\neg
    f_{x}\})$: Similar to (1), just replacing $c$ with $c^{-}$.
\end{enumerate}
It remains to show that exactly one of $f_{x}$ and $\neg f_{x}$ lies in
$A_{q}$, but this is obvious otherwise $c \in H(\emptyset)$ as $H$ is with
connectives.
\end{proof}

\begin{theorem}\label{thm:p-completion}
If $p=\langle H,f,f^-,c\rangle$ is a loopless (hence consistent)
$p$-dialectical system with connectives, then $A_p$ is a completion.
\end{theorem}

\begin{proof}
Let $p=\langle H,f,f^{-},c\rangle$ a $p$-dialectical system with connectives
where $c\not\in H(\emptyset)$ (in such a way that something is not
derivable). Let $f_u=x$ and $f_v=\lnot x$ and without loss of generality
assume $u<v$. Suppose that $x\not\in A_p$; then $c\in H(L(u)\cup\{x\})$, and
by property (6) of definition 3.1, we have $x\rightarrow c\in H(L(u))$, from
which $\lnot x\in H(L(u))$. Suppose now that also $c\in H(L(v)\cup\{\lnot
x\})$, and therefore by the same argument  $x\in H(L(v))$. But there will be
a stage s such that for all $t\geq s$ we will have $L(v)=L_t(v)$. Moreover,
since $L(u)\subseteq L(v)$ and $H$ is an algebraic closure operator we can
assume that a $t$ is big enough to have $L(v)\subseteq L_t(v)\subseteq
H_t(L_t(v))$ from which $L(u)\subseteq H_t(L_t(v))$, giving that both $\lnot
x$ and $x$ belong to $H_t(L_t(v))$ and therefore $c \in H_s(L_s(v))$ for some
$s \ge t$, giving that $L_s(v)$ must change after $t$: contradiction.
\end{proof}

\subsection{Comparing dialectical, $q$-dialectical, and $p$-dialectical
completions} We now consider the relationships under inclusion of the various
systems with connectives.

An immediate consequence of Theorem~\ref{thm:from-d-to-p} is the following:

\begin{corollary}\label{cor:from-d-compl-to-q-compl}
Every dialectical completion is also a $p$-dialectical completion.
\end{corollary}

\begin{proof}
The proof of Theorem~\ref{thm:from-d-to-p} shows that starting from a
dialectical system $d=\langle H, f, c\rangle$, with $H(\emptyset)$ infinite
then one can build a $p$-dialectical system $p$ with the same $H$, and the
same $c$, $p$ has connectives as $H$ does. On the other hand, the condition
that $H(\emptyset)$ be infinite is granted by the fact that $H$ has
connectives, and thus, for instance, if $x \in H(\emptyset)$ then $x \wedge x
\in H(\emptyset)$ as well.
\end{proof}

\begin{theorem}\label{thm:q-with-connectvs-is-dialectical}
If $(q,\alpha)$ is a consistent loopless $q$-dialectical pair, with
$q=\langle H, c, c^{-}, f, f^{-}\rangle$ a $q$-dialectical system with
connectives, and $\alpha$ a good approximation to $H$, then $A_{q}^{\alpha}$
is a dialectical completion.
\end{theorem}

\begin{proof}
Suppose that  $(q,\alpha)$ is a loopless $q$-dialectical pair, $q=\langle H,
c, c^{-}, f, f^{-}\rangle$ is a $q$-dialectical system with connectives, $c
\notin H(\emptyset)$ and $\alpha$ is a good approximation to $H$. Then
$A^\alpha_q$ is a completion by Theorem~\ref{thm:q-implies-completion}, and
thus $\neg c^{-} \in A_{q}^{\alpha}$: let $u$ be such that $f_{u}=\neg
c^{-}$, hence $r(u)=\langle \neg c^{-}\rangle$, and let $t_{0}$ be the least
stage such that $L(u+1)$ has reached limit already, $\neg c^{-} \le t_{0}$
(thus each $s\ge t_{0}$ has an axiom $\langle \neg c^{-}, \{\neg
c^{-}\}\rangle \in H_{s}$), and $c \in H_{t_{0}}(\{c^{-}, \neg c^{-}\})$.
Suppose now that $s\ge t_{0}$ is a stage such that $c^{-} \in
H_{s}(L_{s}(v))$ with $v> u$. But $H_{s}$ is an algebraic closure operator,
as $\alpha$ is good: therefore  $\neg c^{-}\in H_{s}(L_{s}(v))$ since $\neg
c^{-} \in L_{s}(v)$. This gives $\set{c^{-}, \neg c^{-}} \subseteq
H_{s}(L_{s}(v))$, hence $c\in H_{s}(H_{s} (L_{s}(v))\subseteq
H_{s}((L_{s}(v))$.  It is then clear that starting from $t_0$, the
$q$-dialectical procedure behaves as a dialectical procedure, since $f^-$ no
longer plays any role.

Let $v$ be the greatest slot such that for every $s \ge t_0$,
$L_s(v)=L_{t_0}(v)$ (clearly $v>u$; such a maximum exists since at $t_{0}$
almost all $r(v)$ are empty), and let $d=\langle H, g,c \rangle$ be the
dialectical system where $g$ is defined as follows. First fix a strictly
increasing computable sequence $z_{0}<z_{1}< \cdots$ of elements of
$H(\emptyset)$. Then
\begin{itemize}
\item if $v'<v$ then let
\[
g_{v'}=
\begin{cases}
f_{v'}, &\text{if $r(v')=\langle \mbox{ }\rangle$},\\
\rho(v'), &\text{otherwise
and $\rho(v')\notin \{g_{v''}: v''<v'\}$},\\
\min  z_{i}\notin \{g_{v''}: v''<v'\}, &\text{otherwise};
\end{cases}
\]
\item
if $v'\ge v$ then let
\[
g_{v'}=
\begin{cases}
f_{v'}, &\text{if $f_{v'} \notin \{g_{v''}: v''<v\}$},\\
\min  z_{i}\notin \{g_{v''}: v''<v'\}, &\text{otherwise}.
\end{cases}
\]
\end{itemize}
Then $g$ is a computable permutation and by the above remarks, it is easy to
see that $A^{\alpha}_{q}=A_{d}$.
\end{proof}

\begin{theorem}
If $p=\langle H, f, f^{-},c\rangle$ is a loopless $p$-dialectical system with
connectives, in which $f^-=\neg$, then $A_p$ is both a dialectical
completion, and a $q$-dialectical completion.
\end{theorem}

\begin{proof}
Let $p=\langle H, f, f^{-},c\rangle$ be a $p$-dialectical system with
connectives in which $f^-=\neg$ and $c$ is a contradiction. Let $d=\langle H,
f, c \rangle$: we claim that $A_{p}=A_{d}$. Let us use the superscripts $p$
and $d$, to distinguish the relevant parameters of $p$ and $d$, respectively.
We will prove by induction on $u$ that
\begin{itemize}
\item  if $r^{d}(u)=\langle f_{u} \rangle$ then $r^{p}(u)=\langle
    f_{u}\rangle$, and if $r^{d}(u)=\langle \mbox{ } \rangle$ then
    $r^{p}(u)=\langle f_{u}, \neg f_{u} \rangle$;

\item for every $v\le u$, if $r^{d}(v)=\langle \mbox{ } \rangle$ then $\neg
    f_{v} \in H(L_{d}(v))$.
\end{itemize}
Notice that from this and the fact that $A_{q}$ and $A_{p}$ are completions,
it easily follows that $A_{p}=A_{d}$.

Case $u=0$ (base of the induction).  This case easily follows from the
assumptions and the basic definitions.

Suppose that the clam is true of $u$. If $c \in H(L_{p}(u+1)\cup
\{f_{u+1}\})$ then $c \in H(L_{d}(u+1) \cup \{\neg f_{v}: v\le u \,\&\,
r^{d}(v)=\langle \mbox{ } \rangle\} \cup \{f_{u+1}\})$. By the inductive
assumption, $\{\neg f_{v}: v\le u \,\&\, r^{d}(v)=\langle \mbox{ }
\rangle\}\subseteq H(L_{d}(u+1))$, hence $c \in H(L_{d}(u+1)\cup
\{f_{u+1}\})$. This shows that if $r^{d}(u+1)=\langle f_{u+1} \rangle$ then
$r^{p}(u+1)=\langle f_{u+1}\rangle$. The claims that if $r^{d}(u+1)=\langle
\mbox{ } \rangle$ then $r^{p}(u+1)=\langle f_{u+1}, \neg f_{u+1} \rangle$,
and if $r^{d}(u+1)=\langle \mbox{ } \rangle$ then $\neg f_{u+1} \in
H(L^{d}(u+1))$, come straight from the definitions.

The remaining claim (i.e. $A_p$ is a $q$-dialectical completion) follows from
the following lemma.
\begin{lemma}
For every dialectical completion $A_d$ there exists a loopless
$q$-dialectical pair $(q, \alpha)$ such that $A_d=A^\alpha_q$.
\end{lemma}
\begin{proof}
Let $d=\langle H, f, c\rangle$ be a consistent dialectical system with
connectives. By Lemma~\ref{lem:goodapproximation} let $\alpha$ be a good
approximation to $H$; let $c^-$ be $c \wedge c$ (thus $c \in H(\{c'\})$);
finally let $f^-$ be any proposing function. Notice that $q=\langle H, f,
f^-, c, c^- \rangle$ is a (proper) $q$-dialectical system as $c \ne c^-$. We
claim that $A_d=A^\alpha_q$. This follows from the fact that $c^-$ does not
play any role in the $q$-dialectical procedure, as if $c^- \in H_s(X)$, then
by goodness of the approximation, we also have $c\in H_s(X)$ since $c \le \in
H_s(\{c^-\})\subseteq H_s(H_s(X))\subseteq H_s(X)$.
\end{proof}

\end{proof}

\section{$p$-dialectical sets and degrees}

The characterizations of the Turing degrees of the dialectical sets and of
the $q$-dialectical sets has been given in \cite{trial-errorsII}:

\begin{lemma}\label{lem:degrees}
The Turing degrees of the dialectical sets, and of the $q$-dialectical sets,
are exactly the c.e. Turing degrees.
\end{lemma}

\begin{proof}
See \cite{trial-errorsII}.
\end{proof}

Let us now  consider the case of $p$-completions. If $T$ is a formal theory
with set of theorems $\Thm_T$, and $d$  is a dialectical system with
connectives such that $H(\emptyset)= \Thm_T$, then we say that $d$ is a
\emph{dialectical system for $T$}. If $d$ is a consistent dialectical system
for $T$, and $T$ is consistent, then $A_d$ is a completion of $T$. Let us
consider a propositional calculus with propositional atoms $\set{p_{i}: i\in
\omega}$: by codes, we assume that this set coincides $\omega$. Given a set
$A \subseteq \omega$, let $T_{A}$ be the propositional calculus, obtained by
adding to the classical propositional calculus the axioms $\set{p_{i}: i \in
A}$. The following is due to \cite{Bernardi}.

\begin{lemma}\label{lem:dial+connectives}
For every c.e. $A$ there exists a dialectical system $d=\langle H,f,c
\rangle$ for the theory $T_A$, such that:
\begin{enumerate}
\item $A\leq_mA_d$,
\item $A_d\leq_{tt} \Thm_{T_A}$,
\item $\Thm_{T_A}\leq_{tt} A$,
\end{enumerate}
and therefore $A_d \equiv_{tt} A$.
\end{lemma}

\begin{proof}
See \cite{Bernardi}.

\end{proof}

\begin{corollary}
The c.e. Turing degrees coincide with the degrees of $p$-completions, and
with the degrees of $p$-dialectical sets.
\end{corollary}

\begin{proof}
If $A$ is a c.e. set then by the above lemma there is a dialectical
completion $A_d$ with the same $tt$-degree as $A$. But every dialectical
completion is a $p$-completion by
Corollary~\ref{cor:from-d-compl-to-q-compl}, and thus every c.e. Turing
degree contains a $p$-completion. On the other hand every $p$-dialectical set
is also $q$-dialectical, thus  by Lemma~\ref{lem:degrees} we have that the
degree of any $p$-dialectical set is c.e.
\end{proof}

\section{A $p$-dialectical completion, which is
neither a dialectical completion, nor a $q$-dialectical completion }

In the following $T$ is taken to be Peano Arithmetic (assumed to be sound).

The following lemma has been known to logicians for many years already, and a
proof-theoretic proof can be found in, or at least worked out from,
Smory\'nski \cite[p.~362]{Smorynski:fifty}. This proof uses a version of the
fixed point theorem originally due to Kent \cite{Kent}. Notably it is
based on Rosser's method of comparison of witnesses and includes a
relativized proof predicate as in Kreisel-Levy Essential Unboundedness
Theorem, asserting that a certain formula is derivable from a true formula of
a certain fixed complexity (\cite[p.~362]{Smorynski:fifty})

We propose a purely computability-theoretic proof, which looks perhaps
simpler than \cite{Smorynski:fifty}. Being a $\Sigma_{n}$ ($\Pi_{n}$)
sentence means of course being provably equivalent in $T$ to a sentence which
is synctactically $\Sigma_{n}$ ($\Pi_{n}$).

\begin{lemma}\label{lem:fundamental}
For every $n\ge 1$, there exists a sentence $\psi \in \Sigma_{n+1}$ such
that, for every $\phi \in  \Delta_{n+1}$, if $\not \vdash_{T} \phi$, then $\not
\vdash_{T} \psi \rightarrow \phi$ and $\not \vdash_{T} \neg \psi \rightarrow
\phi$.
\end{lemma}

\begin{proof}
Suppose $S$ is the set of all $\Delta_{n+1}$-sentences. We need a $\psi$
such that, for all $\phi \in S$, if $T+\neg \phi$ is consistent, then $\psi$ is
independent of $T+ \neg \phi$.

Recall that $S$ is c.e., so let ${\phi_0, \phi_{1}, \ldots}$ be a recursive
enumeration of $S$. Let $\textrm{Dim}_T$ denote the standard provability
predicate, expressing, via codes, whether a given number is a proof of a
given formula. For each $j$, we define the function $f_j$ as follows: On
input $s$, search for the least $i$ such that either
\begin{enumerate}
\item[(a)] $\textrm{Dim}_T(s,\ulcorner \Phi_j^{\emptyset^{(n)}}(0)=1
\rightarrow \phi_i\urcorner) \wedge \neg \phi_i$, or
\item[(b)] $\textrm{Dim}_T(s, \ulcorner\neg
\Phi_j^{\emptyset^{(n)}}(0)=1 \rightarrow
\phi_i\urcorner ) \wedge \neg \phi_i$,
\end{enumerate}
and define
\[
f_j(s)=
\begin{cases}
1 &\text{if $i$ is found, and (a) holds},\\
0 &\text{if $i$ is found, and (b) holds,}\\
\uparrow &\text{if no such $i$ is found}.
\end{cases}
\]
By the Relativized Parameter Theorem, $f_{j}=\Phi^{\emptyset^{(n)}}_{h(j)}$,
for some computable function $h$; and let $g$ be a computable function so
that
\[
\Phi_{g(j)}^{\emptyset^{(n)}}(x)=
\begin{cases}
\uparrow &\text{if $f_j$ has empty domain},\\
1 &\text{if the first value of $f_j$ is $1$
         (i.e. $f_j(m)=1$ where $m$ is the least}\\
\mbox{} &\text{number in the domain of $f_j)$},\\
0 &\text{if the first value of $f_j$ is $0$}
\end{cases}
\]

In the following, we often identify statements relative to $f_j$ or
$\Phi_j^{\emptyset^{(n)}}$  with their formal arithmetical translations. Let
$e$ be a fixed point for $g$. That is:
\[
\Phi_e^{\emptyset^{(n)}}=\Phi_{g(e)}^{\emptyset^{(n)}}.
\]
Let $\psi$ be the sentence which says that $\Phi_e^{\emptyset^{(n)}}(0)=1$.

Claim:
\begin{enumerate}
\item If $s$ is a proof from $\psi$ to $\phi_i$ for some $i$, then
    $T$ proves $\phi_i$ (and thus $\phi_i$ is true);
\item If $s$ is a proof from $\neg \psi$ to $\phi_i$ for some $i$,
    then $T$ proves $\phi_i$ (and thus $\phi_i$ is true).
\end{enumerate}

\begin{proof}
We induct on $s$, assuming the lemma for all $t<s$. Since the claim is true
for all $t<s$, $f_e(t)$ diverges for all such $t$. Note that $T$ can prove
that $f_e(t)$ diverges for all $t<s$. For each $t<s$, $T$ can determine if
$t$ is a proof of $\psi \rightarrow \phi_i$ or $\neg\psi \rightarrow \phi_i$
for some $i$. If not, then clearly $f_e(i)$ diverges. If it is a proof of
that form, then by our inductive hypothesis, $T$ also proves $\phi_i$. Thus,
$T$ proves that $f_e(t)$ diverges, since $\neg\phi_i$ is a condition for
convergence of $f_e(t)$.

\begin{enumerate}
\item If $s$ is a proof from $\psi$ to $\phi_i$, then $s$ is a proof of
    $\phi_i$ from $\Phi_e^{\emptyset^{(n)}}(0)=1$. $T$ can argue: Either
    $\phi_i$ is true or $\phi_i$ is false. If $\phi_i$ is false, then
    $f_e(s)$ converges to $1$. This means that
    $\Phi^{\emptyset^{(n)}}_{g(e)}(0)=1$. But then this means that $\phi_i$
    is true (from the proof $s$). Thus, $T$ has proved that $\phi_i$  is
    true.
\item If $s$ is a proof from $\neg \psi$ to $\phi_i$, then $s$ is a proof
    of $\phi_i$ from $\neg \Phi_e^{\emptyset^{(n)}}(0)=1$. $T$ can argue:
    Either $\phi_i$ is true or $\phi_i$ is false. If $\phi_i$ is false,
    then $f_e(s)$ converges to $0$. This means that
    $\Phi^{\emptyset^{(n)}}_{g(e)}(0)=0$. But then this means that $\phi_i$
    is true (from the proof $s$). Thus, $T$ has proved that $\phi_i$ is
    true.
\end{enumerate}
Hence, for any $i$ such that $\phi_i$ is not a theorem of $T$, there can be
no proof in $T$ of $\psi \rightarrow \phi_i$, or $\neg \psi \rightarrow
\phi_i$.
\end{proof}
\end{proof}

\begin{remark}
Notice that in the previous lemma, the sentence $\psi$ associated with  the
set of all $\Delta_{n+1}$-sentences  is $\Sigma_{n+1}$.
\end{remark}

A class $\mathcal{A}$ of $\Delta^{0}_{2}$ sets is called \emph{computable} if
there is a $\Delta^{0}_{2}$ predicate $A(e,x)$ such that $\mathcal{A} =
\{V_{e}: e \in \omega\}$, where
\[
V_{e}=\{x: A(e,x)\}.
\]
If $\{A(e,x,s): s \in \omega\}$ is a computable approximation to $A(e,x)$,
i.e. $\lim_{s} A(e,x,s)= A(e,x)$ for every $x$, then we let $V_{e,s}(x)=
A(e,x,s)$.

\begin{theorem}\label{thm:main}
If  a class $\mathcal{A}$ of $\Delta^{0}_{2}$ sets is computable, then there
is a $p$-dialectical system $p$ with connectives such that $A_{p}$ is a
completion of Peano Arithmetic and $A_{p} \notin \mathcal{A}$.
\end{theorem}

\begin{proof}
Suppose we are given a computable class of $\Delta^{0}_{2}$ sets
$\mathcal{A}=\{V_{e}: e \in \omega\}$.  We want to build a $p$-dialectical
system $p=\langle K, f, f^-, c \rangle$ with connectives, satisfying the
requirements
\[
N_e: A_p \ne V_e,
\]
and such that $A_{p}$ is a completion of Peano Arithmetic. Let again $T$
denote Peano Arithmetic, and let $H$ be the enumeration operator given by
\[
H=\{\langle x, D\rangle: D \vdash_{T} x\}.
\]
Via a suitable G\"odel numbering, throughout the proof, numbers should be
thought of as sentences of the language of $T$. We choose $c$ to be the usual
contradiction $0=1$.

The construction is by stages. At the end of stage $s$ we will have defined a
finite set $\Ax_{s}$ of axioms to be added to the axioms of $T$, and finite
approximations $f^s$, $f^-_s$ to computable functions $f$, $f^-$,
respectively, so that $f=\bigcup_s f^s$, $f^-=\bigcup_s f^-_s$, and
$\Ax=\bigcup_s \Ax_s$ is a c.e. set. In order to define a $p$-dialectical
system, we will have also to specify a suitable enumeration operator $K$:
since the construction is computable, the theory $S_{\infty}$ obtained by
adding all axioms $\bigcup_{s} \Ax_{s}$ to those of $T$ is a c.e. extension
of $T$, and we will let
\[
K=\{\langle x, D\rangle: D \vdash_{S_{\infty}} x\}.
\]
\begin{lemma}\label{lem:Kobvious}
$K$ is a algebraic closure operator with connectives, and for every set $X$, $H(X)
\subseteq K(X)$.
\end{lemma}

\begin{proof}
Immediate.
\end{proof}

By Lemma~\ref{lem:fundamental}  let $\Gamma$ be a computable function which
with every finite set $S$ of sentences associates a sentence $x$ such that
\[
(\forall S' \subseteq S)[c \notin H(S') \Longrightarrow x
\notin H(S') \,\&\, \neg x \notin H(S')].
\]
We say in this case that $x$ has been chosen to be \emph{independent} of
every such $S'$. In the rest of the proof, we will distinguish between
$K$-consistency (i.e. consistency in $S_\infty$: a set $X$ is
\emph{$K$-consistent} if $c \notin K(X)$) and $H$-consistency (i.e.
consistency in $T$: a set $X$ is \emph{$H$-consistent} if $c \notin H(X)$).

\subsection*{The strategy to meet $N_e$}
We outline the construction and the strategy to meet the requirement $N_e$,
and we describe what our desired $p$-dialectical system should achieve. In
addition to $f^s, f^-_s$, throughout the construction we use several
computable parameters, which are modified stage by stage: $x_{e,s}$, $\hat
\rho_s(u)$, $\hat r_s(u)$, $A^s$. In particular $A^s$ stands for a finite
set, such that, for every $u$, $A(u)=\lim_s A^s(u)$ exists; the parameters
$x_{e,s}$, $\hat \rho_s(u), \hat r_s(u)$ will be such that $x_e=\lim_s
x_{e,s}$, $\lim_s \hat \rho_s(u)=\hat \rho(u)$, and $ \lim_s \hat r_s(u)=\hat
r(u)$ exist, and $\rho_s(u)$ will coincide with $\rho(u)$ of the
$p$-dialectical system we are aiming at; moreover $A(x_e)=A_p(x_e)$ for every
$e$.

We reserve the two slots $3e, 3e+1$ to attack and satisfy $N_{e}$. The action
may take place at several different stages: at each stage $s$ we denote by
$C=C_{s}$ the set consisting of all (finitely many) Boolean combinations of
the sentences corresponding to the numbers so far mentioned and used in the
construction.

The first time at which we attack requirement $N_e$ we let $f_{3e}=y_{e}$,
$f_{3e+1}=x_{e}$, where $C:=C_{s}$, and
\begin{align*}
&x_{e} = \Gamma(C)\\
&y_{e} = \Gamma (C \cup \{x_{e}\}).
\end{align*}
We then execute the following cycle, which starts with $k=0$,
$y_{e}(0)=y_{e}$:
\begin{enumerate}
  \item wait until the least stage $t>s$ such that
  $x_{e} \in V_{e,t}$, then add the axiom $\neg (y_{e}(k) \wedge
      x_{e})$ in $\Ax$; extract $x_{e}$ from $A$ (i.e., define
      $A^t(x_e)=0$); go to (2) with $s:=t$;
  \item wait until the least stage $t>s$ such that
  $x_{e} \notin V_{e,t}$ then add the axiom $\neg y_{e}(k)$ in $\Ax$,
      define $y_{e}(k+1)=f^{-}(y_{e}(k))=\Gamma (C)$;  add $x_{e}$ into $A$
      (i.e., define $A^t(y_{e}(k+1))=1$ and $A^t(x_{e})=1$); go to (1) with
      $s:=t$ and $k:=k+1$.
\end{enumerate}

\subsubsection*{Outcomes of the strategy}
The cycle eventually stops since $V_e(x_e)$ may change only finitely many
times, and eventually $A(x_e)\ne V_e(x_e)$. Having in mind the
$p$-dialectical system which we want to build and its characterizing
parameters $\rho, r, L$, this cycle must be viewed as our attempt to build a
stack $r(3e)$ of which the number $y_{e}(k)$ becomes the top  when it is
appointed; similarly, when $x_{e}$ is initially appointed we have $r(3e+1)=
\langle x_{e} \rangle$. Our intended goal is that if $L(3e)$ is
$K$-consistent (i.e. $c \notin K(L(3e))$) and we add the axiom $\neg y_e(k)$
in $\Ax$ then $y_e(k)$ will be discarded by the $p$-dialectical procedure (as
$L(3e)\cup \{y_e(k)\}$ is not $K$-consistent) and it will be replaced by
$y_e(k+1)$  so as to momentarily have $L(3e) \cup \{y_e(k+1)\}$
$K$-consistent; so the $p$-dialectical procedure will put back $x_{e}$ as a
thesis. If $L(3e) \cup \{y_{e}(k)\}$ is $K$-consistent, and we add the axiom
$\neg (y_e(k) \wedge x_{e})$ in $\Ax$, then the $p$-dialectical procedure
keeps $L(3e) \cup \{y_{e}(k)\}$ $K$-consistent and discards $x_{e}$ as a
thesis. This process is repeated as many times as are needed to diagonalize
$A(x_{e})\ne V_e(x_{e})$. Use of the function $\Gamma$ in choosing each
$y_{e}(k)$ and $x_{e}$ allows us to conclude that $y_{e}(k)$ does not clash
with $L(3e)$ to derive $c$, and $x_{e}$ does not clash with $L(3e+1)$ to
derive $c$, for any reasons other than those due to which we add axioms in
$\Ax$, i.e.\ in order to diagonalize $A$ against $V_{e}$. If our
$p$-dialectical system is able to mirror faithfully the cycle for $N_e$ as
described, then $A_p(x_e)= A(x_e)$ and thus $A_p(x_e)\ne V_e(x_e)$.

\subsubsection*{Other issues in defining $f$ and $f^{-}$} We must also come up
with $f$ being a permutation of $\omega$, and with $f^-$ being total.

At non-zero even stages we take care of surjectivity of $f$, by picking the
least available slot $u$ with $u=3k+2$ for some $k$, and the least $x$ which
has not as yet been proposed by $f$: we define $f_{u}=x$ and if $f^{-}(x)$ is
as yet undefined then we define $f^{-}(x)=a_0$, where $a_0 \in H(\emptyset)$.
Note that injectivity of $f$ is immediate by construction.

At non-zero even stages we also pick the least $z$ such that $f^-(z)$ has not
as yet been defined, and we put $f^-(z)=\Gamma(C)$ with $C$ evaluated at
these stages.

\subsection*{The construction}
The modifications imposed on a parameter during a stage of the construction
will determine the final value of the parameter at the end of the stage. It
is understood that a parameter which is not explicitly modified at a stage
$s$, maintains, at the end of stage $s$, the same value as the one it
possessed at the beginning of stage $s$. Throughout a stage $s+1$, if a
parameter is mentioned without specifying any stage of approximation, then it
is understood to be evaluated with the value it possessed by the end of stage
$s$. When we apply the function $\Gamma$ at stage $s+1$, without loss of
generality, we may assume that $\Gamma$ picks an element which is different
from all numbers so far mentioned in the construction, in particular from
every number already in the domain of $f^{-}$: otherwise, as in Craig's trick
for computable axiomatizability of c.e. theories, take as many iterations $w
\wedge w \wedge \dots$ of the conjunctive connective $\wedge$ on the value
$w$ provided by $\Gamma$ as are needed to achieve this goal.

\subsubsection*{Step $0$}
Choose a c.e. injective sequence $a_{0}< a_{1}< \cdots$ in $H(\emptyset)$,
and define $f^{-}_{0}(a_{i})=a_{i+1}$, for every $i$. Define also
$f^{0}=A^{0}=\emptyset$ and $\Ax_0=T$. All the other parameters are
undefined.

\subsubsection*{Step $s+1$, odd}
We say that a requirement $N_e$\emph{ requires attention at $s+1$}, if either
\begin{enumerate}
  \item[(r1)] $x_{e}$ is not defined; or
  \item[(r2)] either $x_{e} \in V_e$ but no axiom $\neg (\hat{\rho}(3e)
      \wedge x_{e})$ lies in $\Ax$, or $x_{e} \notin V_e$ and
      $\neg \hat{\rho}(3e) \in \Ax$.
\end{enumerate}

Consider the least $e$ such that $N_e$ requires attention, and take action
accordingly:

\begin{enumerate}
  \item[(a1)] if $N_e$ requires attention through (r1), then define (where
      $C$ is the current approximation to the set $C$ as in the above
      description of the strategy for $N_e$)
\begin{align*}
&x_{e,s+1} = \Gamma(C)\\
&y_{e,s+1} = \Gamma (C \cup \{x_{e,s+1}\}).
\end{align*}
Define  $f^{-}_{s+1}(x_{e,s+1})=\Gamma (C \cup \{x_{e,s+1}, y_{e,s+1}\})$.
\item[(a2)] if $N_e$ requires attention through (r2) then we further
      distinguish the following two cases:
      \begin{enumerate}
        \item[(a21)] if $x_e \in V_e$ then add the axiom $\neg
            (\hat{\rho}(3e) \wedge x_{e})$ in $\Ax_{s+1}$; let
            \[
            \hat
            r_{s+1}(3e+1)=r(3e+1) \widehat{\mbox{ }} \langle
            f^{-}_{s+1}(x_{e,s+1}) \rangle,
            \]
            and $A^{s+1}(x_e)=0$;
        \item[(a22)] if $x_{e} \notin V_e$, then add an axiom $\neg
            \hat{\rho}(3e)$ in $\Ax_{s+1}$; define
            $f^{-}_{s+1}(\hat{\rho}(3e))=\Gamma (C)$ if
            $f^{-}(\hat{\rho}(3e))$ is undefined; let
            $\hat{\rho}_{s+1}(3e)=f^{-}_{s+1}(\hat{\rho}(3e))$; let
\[
\hat r_{s+1}(3e)=\hat r(3e) \widehat{\mbox{ }} \langle \hat{\rho}_{s+1}(3e)\rangle.
\]
Let also $A^{s+1}(x_e))=1$.
     \end{enumerate}
\end{enumerate}
\emph{Resetting.} For all relevant $u>3e+1$ i.e. $u$ of the form $3e'$ or
$u=3e'+1$, let $\hat r_{s+1}(u)=\langle { }\rangle$, and consequently
$\hat\rho_{s+1}(u)$ be undefined. Add the axioms $\neg \hat{\rho}(u)$ in
$\Ax_{s+1}$.

On all remaining $u$ which are different from the $x_i$ that are still
defined at the end of this stage, define $A^{s+1}(u)=0$.

\subsubsection*{Step $s+1>0$, even}
Let $u$ be the least available slot of the form $u=3e+2$, and let $x$ be the
least number such that $x \notin \range(f)$: define $f_{u}^{s+1}=x$ and
$f^-_{s+1}(x)=a_0$ if $f^-(x)$ has not been already defined; otherwise, let
$i$ be the greatest number such that $(f^-)^i(x)$ is already defined, and
define $f^-_{s+1}((f^-)^i(x))=a_0$.

Pick the least $z$ such that $f^-(z)$ is not as yet defined, define
$f^-_{s+1}(z)=\Gamma (C)$.

\subsection*{The verification}
The following lemma is an easy consequence of the construction.

\begin{lemma}
The function $f$ is bijective and the function $f^{-}$ is acyclic. Hence $p=
\langle K, f, f^{-}, c\rangle$ is a $p$-dialectical system.
\end{lemma}

\begin{proof}
$f^{-}$ is acyclic because we define it through $\Gamma$ which picks at each
stage numbers not in the domain of $f^{-}$ and because of the way we have
arranged things when we define or have defined at $0$, $f^{-}(v)=a_i$ for
some $i$. The rest of the claim is obvious by Lemma~\ref{lem:Kobvious}.
\end{proof}

Let $S_{\infty}(=K(\emptyset))$ be the c.e. extension of Peano Arithmetic,
having, as additional axioms, the axioms added during the construction.
Define the \emph{entry stage} of a number $v$ which ever appears in a string
$\hat{r}(u)$ with $u \in \{3e, 3e+1: e\in \omega\}$ to be the least $s$ at
which $v$ enters the range of $f$ or $f^{-}$. Next we define the entry stage
of an axiom (one of the axioms added during the construction): the
\emph{entry stage} of an axiom $\neg v$ is the entry stage of $v$, and that
of an axiom $\neg(v \wedge z)$ is the entry stage of $v$.
In the verifications below, it will be useful to keep in mind that if $X, Y$
are sets such that $X\subseteq \Ax$, and $Y$ is $K$-consistent, then $Y \cup
X$ is $H$-consistent.

\begin{lemma}\label{lem:first-limit}
Each requirement acts finitely often. In particular, for every $e$ $\lim_{s}
x_{e,s}=x_{e}$ exists, for every $u \in \{3e, 3e+1: e \in \omega\}$,
$\lim_{s} \hat r_{s}(u)=\hat r(u)$ exists, and thus $\lim_{s} \hat
\rho_{s}(u)=\hat \rho(u)$ exists; finally, $\lim_s A_s(x_e)=A(x_e)$ exists
and $A(x_e)\ne V_e(x_e)$.
\end{lemma}

\begin{proof}
Assume inductively that each $N_i$, $i<e$, eventually stops acting. After
every $N_i$, with $i<e$ has ceased to act, there is a least stage $t_0$ such
that $N_e$ defines the final value of $x_e$. At that point we keep modifying
$\hat r(3e)$ and $\hat r(3e+1)$, only in response to changes in $V_e(x_e)$,
but as $V_e$ is a $\Delta^0_2$ set, this can happen only finitely many times.
The claim about $A(x_e) \ne V_e(x_e)$ easily follows from the construction.
\end{proof}

The following lemmata intend to explicitly relate the above construction to
the $p$-dialectical system $p=\langle K, f, f^-,c\rangle$. In the lemma and
its proof, $\rho, r, L$ are the parameters associated with the
$p$-dialectical system $p=\langle K, f, f^-, c \rangle$.

\begin{lemma}\label{lem:satisfaction0}
Suppose that $u=3e$ is such that $L(u)$ exists in the limit. Then $x_e=\lim_s
x_{e,s}$ exists. Moreover
assume that $L(u) \cup \{\neg x_{e}\}$ is $K$-consistent; then for every $s$
following the last stage at which $L(u)$ changes,

\begin{enumerate}
\item $c \in K(L(u) \cup \{\hat\rho_{s}(u)\})$ if and only if $\neg
    \hat\rho_{s}(u)$ is among the final axioms of $S_\infty$;

\item if $L(u) \cup \{\hat\rho_{s}(u)\}$ is $K$-consistent then $c \in
    K(L(u) \cup \{\hat\rho_{s}(u)\} \cup \{x_{e}\})$ if and only if $\neg
    (\hat\rho_s(u) \wedge x_{e}) $ is among the final axioms of $S_\infty$.
\end{enumerate}
\end{lemma}

\begin{proof}
Suppose that $u=3e$ satisfies the assumptions. In particular from the
existence of $L(u)$ in the limit it is clear that $x_e=\lim_s x_{e,s}$
exists: in fact, once appointed after $L(u)$ ceases to change, we have that
$x_e$ does not change any more. Notice that in both claims (1) and (2) the
right-to-left implication is trivial. So we need only prove the left-to-right
implication of each equivalence. Let us first consider the first item.

\begin{enumerate}

  \item Suppose that $c \in K(L(u) \cup \{\hat\rho_{s}(u)\})$. This means
      that there is a finite subset $X \subseteq \Ax$ such that $c \in
      H(L(u) \cup \{\hat\rho_{s}(u)\} \cup X)$: assume that $X$ is
      $\subseteq$-minimal with this property. It can not be $X=\emptyset$
      as $L(u)$ is $H$-consistent and thus $\hat\rho_{s}(u)$ would be
      chosen independently of $L(u)$: notice that the value
      $\hat\rho_{s}(u)$ has not been chosen before $L(u)$ stops changing
      because of the resetting procedure at the end of odd stages. Let $a
      \in X$ be of greatest entry stage. Notice that, by minimality, $c
      \notin H(L(u) \cup \{\hat\rho_{s}(u)\} \cup (X \smallsetminus
      \{a\}))$, i.e. the set $L(u) \cup \{\hat\rho_{s}(u)\} \cup (X
      \smallsetminus \{a\})$ is $H$-consistent.

We distinguish the two possible cases due to which $a$ can occur as a new
axiom:

\begin{enumerate}
  \item Case $a= \neg v$, for some $v$. In this case, if $v \notin L(u)
      \cup \{\hat\rho_{s}(u)\}$ then (as $v$ has greatest entry stage) we
      have chosen $v$ to be independent of $L(u) \cup \{\hat\rho_{s}(u)\}
      \cup (X \smallsetminus \{a\})$, contradicting that $c \in H(L(u)
      \cup \{\hat\rho_{s}(u)\} \cup X)$ which implies, by logic, that $v
      \in H(L(u) \cup \{\hat\rho_{s}(u)\} \cup (X \smallsetminus
      \{a\}))$. Hence $v \in L(u) \cup \{\hat\rho_{s}(u)\}$. On the other
      hand, it can not be $v\in L(u)$ since $L(u)$ is $K$-consistent.
      Therefore, $v=\hat\rho_{s}(u)$ as desired.

  \item Case $a= \neg (v \wedge z)$, for some $v, z$.  (Recall that in
      this case the entry stage of $a$ is by definition that of $v$.)
      Then $v \wedge z \in H(L(u) \cup \{\hat\rho_{s}(u)\} \cup (X
      \smallsetminus \{a\}))$, hence $v \in H(L(u) \cup
      \{\hat\rho_{s}(u)\} \cup (X \smallsetminus \{a\}))$: if $v \notin
      L(u) \cup \{\hat\rho_{s}(u)\}$, then $v \notin L(u) \cup
      \{\hat\rho_{s}(u)\} \cup (X \smallsetminus \{a\})$ contradicting
      that $v$ (being of greatest entry stage) has been chosen to be
      independent of $L(u) \cup \{\hat\rho_{s}(u)\} \cup (X
      \smallsetminus \{a\})$. Thus $v \in L(u) \cup \{\hat\rho_{s}(u)\}$.
      If $v \in L(u)$ then $v=\hat\rho(u')$ and $z=\hat\rho(v'+1)$ with
      $v'+1<u$, i.e.\ the pair $v',z'$ refer to a requirement $N_{e'}$
      with $z=x_{e'}$ for some $e'<e$, contradicting that $L(u)$ is
      $K$-consistent. Thus we conclude that $v=\hat\rho_{s}(u)$ and
      $z=x_e$, and thus $v \rightarrow x_e \in H(L(u)\cup (X
      \smallsetminus \{a\}))$. By logic we have that $\neg v \in H(\{\neg
      x_{e} \} \cup L(u) \cup (X \smallsetminus \{a\})$, contradicting
      the fact that $v$ is independent of $L(u) \cup \{\neg x_{e}\} \cup
      (X \smallsetminus \{a\})$, as by assumption $L(u) \cup\{\neg
      x_{e}\}$ is $K$-consistent, and $v$ has greatest entry stage, and
      thus by construction $v$ has entry stage greater than or equal to
      that of $x_e$, but, if equal, the claim holds as well by the way we
      choose $v=y_e$ in (a1) of the construction.
\end{enumerate}

\item Suppose that $c \in K(L(u) \cup \{\hat\rho_{s}(u)\} \cup \{x_{e}\})$.
    As before, let $X \subseteq \Ax$ be a finite set such that $c \in
    H(L(u) \cup \{\hat\rho_{s}(u)\} \cup \{x_{e}\} \cup X)$, and $X$ is
    minimal with this property. As in the previous case, we may assume
    $X\ne \emptyset$ as $L(u) \cup \{x_{e}\}$ is $H$-consistent (being
    $x-e$ appointed after $L(u)$ has reached limit and being $L(u)$
    $H$-consistent) and thus by resetting $\hat\rho_{s}(u)$ is chosen
    independently of this set (the entry stage of $\hat\rho_{s}(u)$ is
    greater than or equal to that of $x_e$: if equal the claim follows by
    the way we choose $v=y_e$ in (a1) of the construction). Let again $a
    \in X$ be of greatest entry stage.

\begin{enumerate}
  \item Case $a= \neg v$, for some $v$. As in (1a), we are forced to
      conclude that $v= \hat\rho_s(u)$, which yields a contradiction the
      fact that $L(u) \cup \{\hat \rho_s(u)\}$ is $K$-consistent.

  \item  Case $a= \neg (v \wedge z)$, for some $v, z$. As in (1b) we can
      argue that $v=\hat \rho_s(u)$, and thus $z=x_{e}$.
\end{enumerate}

\end{enumerate}

\end{proof}

\begin{lemma}\label{lem:satisfaction3}
For every $u$, $L(u)$ exists in the limit; moreover if $u=3e$ then $L(u) \cup
\{x_{e}\}$, $L(u) \cup \{\neg x_{e}\}$ are $K$-consistent.
\end{lemma}

\begin{proof}
The proof is by induction on $u$.

\medskip

\emph{Cases $u=0, 1, 2$}.

Assume first that $u=0$. Then $L(0)=\emptyset$. It is enough to show that
 $\{\neg x_{0}\}$ is $K$-consistent. We show only that
$\{\neg x_{0}\}$ is $K$-consistent: the other case is similar. Notice that
$x_{0}=x_{0,s}$ for ever $s>0$, thus $x_0$ has least possible entry stage.
So, suppose that $c \in K(\{\neg x_{0}\})$ and let $X\subseteq \Ax$ be a
minimal set such that $c \in H(\{\neg x_{0}\} \cup X)$. It can not be
$X=\emptyset$ because $x_0$ is chosen independently of $\emptyset$ as $c
\notin H(\emptyset)$. Let $a \in X$ be of greatest entry stage.

We distinguish the two possible cases due to which $a$ can occur as a new
axiom:

\begin{enumerate}
  \item Case $a= \neg v$, for some $v$. It follows that $v \in H(\{\neg
      x_{0}\} \cup (X \smallsetminus \{a\})$, contradicting that $v$ has
      greatest entry stage and thus it is chosen independent of $\{\neg
      x_{0}\} \cup (X \smallsetminus \{a\})$ which is $H$-consistent.

  \item Case $a= \neg (v \wedge z)$, for some $v, z$. In this case we have
      that $v \wedge z  \in H(\{\neg x_{0}\} \cup (X \smallsetminus
      \{a\})$, and thus $v \in H(\{\neg x_{0}\} \cup (X \smallsetminus
      \{a\})$, contradicting that $v$ has greatest entry stage (at most
      equal to that of $x_{0}$, but this, as in (1b) of the proof of the
      previous lemma does not make things different) and thus it is chosen
      independent of $\{\neg x_{0}\} \cup (X \smallsetminus \{a\})$.
\end{enumerate}

By Lemma~\ref{lem:satisfaction0} this implies also the claim for $u=1, 2$, as
for these slots the $p$-dialectical procedure perfectly mirrors the
construction. Notice that consistency of $L(1)$ follows from
Lemma~\ref{lem:satisfaction0} and the fact that, for the final values of
$\rho$, we never add the axiom $\neg \rho(0)$, and consistency of $L(2)$
follows from the fact that we never add the axiom $\neg (\rho(0) \wedge x_0)$
in which case $\rho(1)=x_e$, or we do add this axiom and thus $\rho(1)=a_0\in
H(\emptyset)$.

\medskip

\emph{Cases $u=3e, 3e+1, 3e+2$}, with $e>0$.

Assume that $u=3e$, with $e>0$. We first observe that $L(3e)$ exists in the
limit, and is $K$-consistent. After $L(3e-1)$ remains unchanged, at
subsequent stages for $r(3e-1)$ we observe the following: if there is
$u<3e-1$ such that there are strings $\sigma, \tau$ with $r(u)=\sigma
\widehat{\mbox{ }} \langle f_{3e-1} \rangle \widehat{\mbox{ }} \tau$ (this
can happen for at most one $u$) then $r(3e-1)=\langle f_{3e-1} \rangle
\widehat{\mbox{ }} \tau$;
on the other hand by the way we define $f^-$, we see that $f_{3e-1}$ never
enters any of the stacks $r(u)$ for $u>3e-1$; the only remaining
possibilities are that either $r(3e-1)$ becomes $\langle f_{3e-1}\rangle$ if
$L(3e-1) \cup \{f_{3e-1}\}$ is $K$-consistent, or $\langle f_{3e-1}, a_0
\rangle$ otherwise. It follows also that $L(3e)$ is $K$-consistent.

It remains to see that $L(3e) \cup\{\neg x_{e}\}$ is $K$-consistent. Assume
that it is not $K$-consistent. Then there is a finite set $X \subseteq \Ax$
such that $c \in H(L(3e) \cup\{\neg x_{e}\} \cup X)$, and $X$ is minimal with
this property. As in the case $u=0$ we can exclude the possibility
$X=\emptyset$. Let $a \in X$ be of greatest entry stage. By an argument
similar to that for the case $u=0$, we conclude that either possible case,
i.e. $a$ is of the form $a=\neg v$ or $a=\neg (v\wedge z)$, leads to a
contradiction.

By Lemma~\ref{lem:satisfaction0}, the claim extends to $3e+1$ and $3e+2$ as
well.
\end{proof}

\begin{lemma}\label{lem:rho}
If $u \in \{3e, 3e+1: e \in \omega\}$ then $\hat{\rho}(u)=\rho(u)$.
\end{lemma}

\begin{proof}
For these $u$ the $p$-dialectical procedure faithfully mirrors the
construction. The only exception is that, by resetting, $\hat r(u)$  may not
coincide with $r(u)$ but is in any case a final segment of $r(u)$, as the
string $r(u)$ keeps records of all proposals made by $f^-$ including those
made even before $L(u)$ has stopped changing, whereas  $\hat r(u)$ is reset
every time $L(u)$ changes. But part of the resetting procedure is adding the
axiom $\neg \hat{\rho}(u)$ every time there is a change in $L_u$. Therefore
each such $\hat{\rho}(u)$ is discarded by the $p$-dialectical procedure, and
after $L(u)$ has reached limit the stack $r(u)$, after a few consecutive
discarding moves, starts to copy $\hat{r}(u)$.
\end{proof}

\begin{lemma}
For every $e$, the $p$-dialectical set $A_p$ is a completion of Peano
Arithmetic which satisfies $A_p \ne V_e$.
\end{lemma}

\begin{proof}
The system $p$ is loopless and so by Theorem~\ref{thm:p-completion} $A_p$ is
a completion. By Lemma~\ref{lem:rho}, it follows that $\rho(3e+1)=\hat
\rho(3e+1)$, and thus by Lemma~\ref{lem:first-limit}, $A(x_e)\ne V_e(x_e)$.
On the other hand, $A_p(x_e)=A(x_e)$. For this we use also that by
Theorem~\ref{thm:independence} $A_p$ consists exactly of the numbers that
eventually occupy $L=\bigcup_u L(u)$, for the final values $L(u)$, and $x_e
\in L$ if and only if $x_e \notin V_e$.
\end{proof}
This ends the prof of Theorem~\ref{thm:main}.
\end{proof}

\begin{corollary}\label{cor:p-q-strange}
There exists a $p$-dialectical system $p$ with connectives
such that $A_p$ is a completion of Peano Arithmetic, and $A_{p}$
is not dialectical.
\end{corollary}

\begin{proof}
Apply the previous theorem, taking $\mathcal{A}$ to be the class of
$\omega$-c.e. sets, which by a result in \cite{trial-errorsII} contains all
dialectical sets, and is known to be a computable class of $\Delta^0_2$ sets.
\end{proof}

\begin{remark}
Notice that the $q$-dialectical system defined in the  proof of
Theorem~\ref{thm:from-p-to-q} need not preserve connectives, even if the
original $H$ does. This is fairly clear from the way $H^*$ is defined: on the
other hand, if the construction of $q$ from $p$ preserved connectives, then
as the result is independent of the approximation to $H^*$, it would be that
$A_d=A^\alpha_q$ where $\alpha$ is a good approximation to $H^*$. But then,
by Theorem~\ref{thm:q-with-connectvs-is-dialectical} $A^\alpha_q$ and thus
$A_p$ would be dialectical. It would follow that every $p$-completion is a
$d$-completion, contrary to Theorem~\ref{thm:main}.
\end{remark}

\end{document}